\documentclass[11pt]{amsart}
\usepackage{enumerate}
\usepackage{graphicx, float, color, verbatim} 
\usepackage{amsthm, amssymb, amsfonts}
\usepackage{caption}
\usepackage{subcaption}

\usepackage{tikz}
\usetikzlibrary{arrows}
\usetikzlibrary{arrows.meta}

\theoremstyle{definition}
\newtheorem{defn}{Definition}[section]
\newtheorem{notate}[defn]{Notation}
\newtheorem{rmk}[defn]{Remark}

\theoremstyle{plain}
\newtheorem{thm}{Theorem}[section]
\newtheorem{conj}[thm]{Conjecture}
\newtheorem{prop}[thm]{Proposition}
\newtheorem{lem}[thm]{Lemma}
\newtheorem{cor}[thm]{Corollary}

\newtheorem{obs}[thm]{Observation}

\renewcommand{\deg}{\mathrm{deg}}
\newcommand{\sgn}{\mathrm{sgn}}
\newcommand{\spn}{\mathrm{span}}

\newcommand{\D}{\mathbb{D}}
\newcommand{\crossingsigno}{$\vcenter{\hbox{\includegraphics[scale=.3]{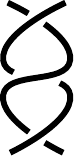}}}$}
\newcommand{\crossingsignt}{$\vcenter{\hbox{\includegraphics[scale=.3]{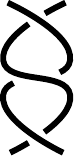}}}$}
\newcommand{\crossingsignth}{$\vcenter{\hbox{\includegraphics[scale=.3]{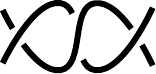}}}$}

\title{On the Qazaqzeh-Chbili-Lowrance conjecture}
\author[S. Goyal]{Shreev Goyal}
\author[J. Kou]{Joshua Kou}
\author[C. Lee]{Christine Ruey Shan Lee}
\author[E. Wu]{Emma Wu}
\author[H. Yang]{Helen Yang}
\author[A. Zhou]{Aaron Zhou}

\address{University of Texas at Austin}
\address{Brown University}
\address{Texas State University}
\email{Christine.rs.lee@gmail.com}
\address{William P. Clements High School}
\address{St. John's School}
\address{California Institute of Technology}

\begin{document}
\maketitle

\begin{abstract}
We compute the minimum crossing numbers for a family of Turaev genus 2 links to verify the Qazaqzeh-Chbili-Lowrance conjecture for the family. 
\end{abstract}

\tableofcontents

\section{Introduction}
In this paper, we investigate the Qazaqzeh-Chbili-Lowrance conjecture \cite{qazaqzeh2025characterizationadequateturaevgenus} for a family of Turaev genus 2 links. The QCL conjecture posits that if the defect between the crossing number of the link and the Jones polynomial is the Turaev genus, then the link is adequate. 
\begin{conj}\cite{qazaqzeh2025characterizationadequateturaevgenus}
A non-split link $L$ is adequate if and only if $\spn (v_L(t)) = c(L)-g_T(L)$.
\end{conj}
Here an adequate link is a link that admits an adequate diagram, see Definition \ref{d.adequate}, $\spn (V_L(t))$ is the degree span of the Jones polynomial of the link $L$ (Definition\ref{d.Mdmd}), and $g_T(L)$ is the Turaev genus of the link $L$ (Definition \ref{d.tg}).

The Turaev genus of a link $L$ is defined by minimizing a diagrammatic quantity $g_T(D)$ over all possible diagrams $D$ of $L$. The Turaev genus is 0 for an alternating link and otherwise difficult to compute for general links. While it is known to be additive with respect to the connected sum for adequate links \cite{Abe-tg}. It is not known at the time of this writing to be additive in general with respect to the connected sum. The relationship to the Jones polynomial is that in general, we have the following inequality \cite{DL}: 
\[ \spn (v_L(t)) \leq  c(L)-g_T(L). \]

It is well known that the adequate diagram of an adequate link realizes the minimum crossing number \cite{Lickorish} and also the Turaev genus \cite{Abe-tg}. Moreover, the span of the Jones polynomial $\spn(V_L(t))$ (Definition \ref{d.Mdmd}) satisfies \cite{Lickorish}
\begin{equation} \label{e.spanecminusg} \spn (v_L(t)) = c(L)-g_T(L). \end{equation}

It is an interesting question as suggested by the QCL conjecture whether equality \eqref{e.spanecminusg} characterizes adequate links. For alternating links, this problem was solved independently by Howie \cite{Howie} and Greene \cite{Greene} , and Kalfagianni \cite{Kalfagianni} characterized adequate links using the colored Jones polynomial. 

In this paper, we investigate this inequality for a family of links parametrized by positive integers $r, s, t, u , v$. A link $L$ in the family admits a diagram $D(r, s, t, -u, -v)$, where each of the parameters $r, s, t, u, v$ indicates the number of half twists in each maximal twist region labeled by $r, s, t, u, v >2$ as shown in an example in Figure \ref{f.diagram}. 

\begin{figure}[H]
\def\svgwidth{.5\columnwidth}
\begingroup%
  \makeatletter%
  \providecommand\color[2][]{%
    \errmessage{(Inkscape) Color is used for the text in Inkscape, but the package 'color.sty' is not loaded}%
    \renewcommand\color[2][]{}%
  }%
  \providecommand\transparent[1]{%
    \errmessage{(Inkscape) Transparency is used (non-zero) for the text in Inkscape, but the package 'transparent.sty' is not loaded}%
    \renewcommand\transparent[1]{}%
  }%
  \providecommand\rotatebox[2]{#2}%
  \newcommand*\fsize{\dimexpr\f@size pt\relax}%
  \newcommand*\lineheight[1]{\fontsize{\fsize}{#1\fsize}\selectfont}%
  \ifx\svgwidth\undefined%
    \setlength{\unitlength}{426.51209553bp}%
    \ifx\svgscale\undefined%
      \relax%
    \else%
      \setlength{\unitlength}{\unitlength * \real{\svgscale}}%
    \fi%
  \else%
    \setlength{\unitlength}{\svgwidth}%
  \fi%
  \global\let\svgwidth\undefined%
  \global\let\svgscale\undefined%
  \makeatother%
  \begin{picture}(1,0.67216084)%
    \lineheight{1}%
    \setlength\tabcolsep{0pt}%
    \put(0,0){\includegraphics[width=\unitlength,page=1]{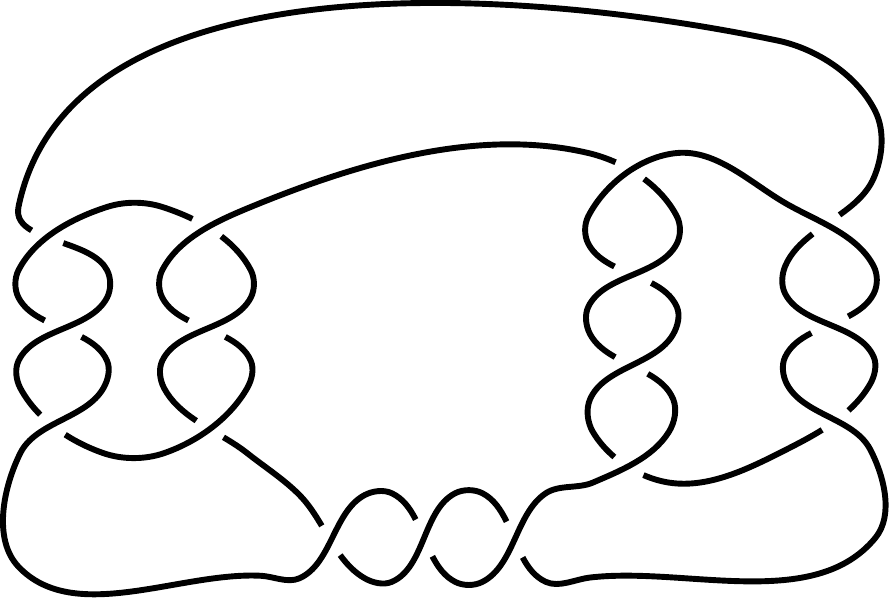}}%
    \put(-0.05121525,0.30037654){\color[rgb]{0,0,0}\makebox(0,0)[lt]{\lineheight{1.25}\smash{\begin{tabular}[t]{l}$r$\end{tabular}}}}%
    \put(0.12814676,0.30037654){\color[rgb]{0,0,0}\makebox(0,0)[lt]{\lineheight{1.25}\smash{\begin{tabular}[t]{l}$s$\end{tabular}}}}%
    \put(0.60628405,0.30749776){\color[rgb]{0,0,0}\makebox(0,0)[lt]{\lineheight{1.25}\smash{\begin{tabular}[t]{l}$t$\end{tabular}}}}%
    \put(0.83610954,0.30546314){\color[rgb]{0,0,0}\makebox(0,0)[lt]{\lineheight{1.25}\smash{\begin{tabular}[t]{l}$u$\end{tabular}}}}%
    \put(0.46580748,0.14167565){\color[rgb]{0,0,0}\makebox(0,0)[lt]{\lineheight{1.25}\smash{\begin{tabular}[t]{l}$v$\end{tabular}}}}%
  \end{picture}%
\endgroup%

\caption{\label{f.diagram} The knot $K(r, s, t, -u, -v) = K(3, 3, 4, -3, -3)$.}
\end{figure}
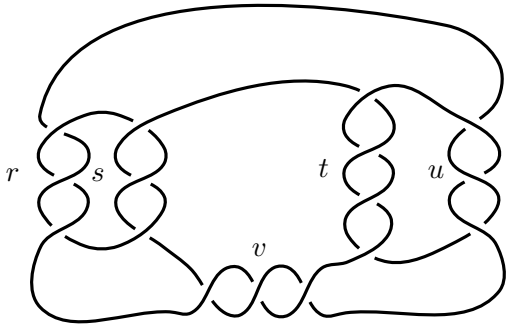

Our convention is to designate a twist region containing $t$ crossings of the form \crossingsigno \ by $t$ and a twist region containing $t$ crossings of the form \crossingsignt \ and \crossingsignth \ by $-t$, reading top to bottom on the page. 

 To compute the Turaev genus, we use results on the leading coefficient of the Jones polynomial of Turaev genus 1 links \cite{DL}. The same result also shows that the family of knots is not adequate. Our result incorporates techniques for finding the crossing number for non-adequate links using the Jones polynomial and the Kauffman 2-variable polynomial. We determine the Turaev genus of the family of links to be 2, and we determine the \textit{defect} $\delta(L)$ of the span of the Jones polynomial from the upper bound $c(L) - g_T(L)$. 
\begin{defn}  
\begin{equation} \label{e.defectdefn} \delta(L) := c(L) - g_T(L) - \spn(V_L(t)).  \end{equation}
\end{defn}

\begin{thm} \label{t.defect}
    Let $L = D(r, s, t, -u, -v)$, then 
\[ \delta(L) = c(L) - g_T(L) - \spn(V_L(t)) = s+v - 2.  \]
\end{thm}
When $s+v > 4$, the fact that $\delta(L)\not= 0$ confirms that for this class of non-adequate links, $\spn(V_L(t)) < c(L)-g_T(L)$. In order to prove the QCL conjecture, one would need to prove the strict inequality for all non-adequate links. We plan to investigate possible sources of counterexamples, for example the family of examples by Manchon \cite{Manchon}, in a future project.

On the additivity of Turaev genus, we end the paper with a result on the possible types of closures of the link diagram in the all-$A$ and all-$B$ state, for future investigations into algorithmic enumeration for counterexamples to the additivity conjecture of the Turaev genus.  This is a sharpening of a result in Lowrance's thesis \cite{Lowrance-thesis}. 

\begin{thm} \label{t.rmoveseffect}
    Let $D$ be a diagram for the link $L$ and let $D'$ be the knot diagram obtained from $D$ after a performing a Reidemeister move on $D_T$ in a small disk $\D^2$. If $g_T(D') < g_T(D)$, then the Reidemeister move must be inverse Type II, or of Type III. Moreover, there are certain conditions on the closures of the link diagram around the move in the A- or B-state: 
    \begin{itemize}
    \item[$\mathrm{(RII)}^{-1}$] If the Reidemeister move is of inverse Type II, then marking $D\cap \partial\D^2$ with $a, b, c, d$, we have, for the initial diagram $D$, the closure as illustrated in Figure \ref{f.r2-1matching}. 
    \begin{figure}[H]
    \centering
    \def \svgwidth{.2\textwidth}
\begingroup%
  \makeatletter%
  \providecommand\color[2][]{%
    \errmessage{(Inkscape) Color is used for the text in Inkscape, but the package 'color.sty' is not loaded}%
    \renewcommand\color[2][]{}%
  }%
  \providecommand\transparent[1]{%
    \errmessage{(Inkscape) Transparency is used (non-zero) for the text in Inkscape, but the package 'transparent.sty' is not loaded}%
    \renewcommand\transparent[1]{}%
  }%
  \providecommand\rotatebox[2]{#2}%
  \newcommand*\fsize{\dimexpr\f@size pt\relax}%
  \newcommand*\lineheight[1]{\fontsize{\fsize}{#1\fsize}\selectfont}%
  \ifx\svgwidth\undefined%
    \setlength{\unitlength}{198.30988366bp}%
    \ifx\svgscale\undefined%
      \relax%
    \else%
      \setlength{\unitlength}{\unitlength * \real{\svgscale}}%
    \fi%
  \else%
    \setlength{\unitlength}{\svgwidth}%
  \fi%
  \global\let\svgwidth\undefined%
  \global\let\svgscale\undefined%
  \makeatother%
  \begin{picture}(1,0.8504008)%
    \lineheight{1}%
    \setlength\tabcolsep{0pt}%
    \put(-0.00343478,0.49308881){\color[rgb]{0,0,0}\makebox(0,0)[lt]{\lineheight{1.25}\smash{\begin{tabular}[t]{l}$a$\end{tabular}}}}%
    \put(-0.00343478,0.31373016){\color[rgb]{0,0,0}\makebox(0,0)[lt]{\lineheight{1.25}\smash{\begin{tabular}[t]{l}$b$\end{tabular}}}}%
    \put(0.91716623,0.50065274){\color[rgb]{0,0,0}\makebox(0,0)[lt]{\lineheight{1.25}\smash{\begin{tabular}[t]{l}$c$\end{tabular}}}}%
    \put(0.91716623,0.32129408){\color[rgb]{0,0,0}\makebox(0,0)[lt]{\lineheight{1.25}\smash{\begin{tabular}[t]{l}$d$\end{tabular}}}}%
    \put(0,0){\includegraphics[width=\unitlength,page=1]{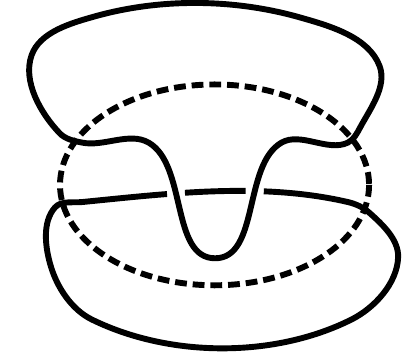}}%
  \end{picture}%
\endgroup%
    
    \caption{\label{f.r2-1matching} Matching of $a, b, c, d$ in $D$ when the Reidemeister move is $\mathrm{(RII)}^{-1}$.}
     \end{figure}    
    \item[$\mathrm{(RIII)}/\mathrm{(\overline{RIII})}$] If the Reidemeister move is of Type III, then the possible closures of the ends of the tangle $D_T$ marked by $a, b, c, d, e, f$ are illustrated in Figure \ref{f.r3closure}. 
    \begin{figure}[H]
    \centering
    \def \svgwidth{.7\textwidth}
    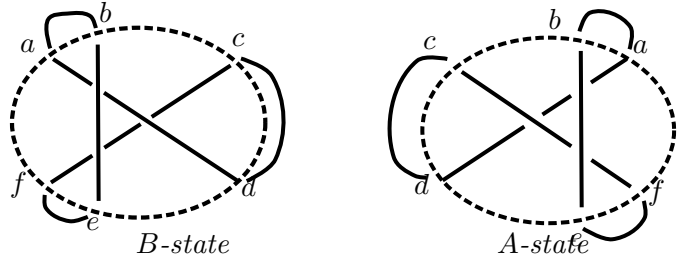 
    \caption{\label{f.r3closure} Left: Matching of $a, b, c, d, e, f$ when the Reidemeister move is $\mathrm{RIII}$. Right: Matching of $a, b, c, d, e, f$ when the Reidemeister move is $\mathrm{\overline{RIII}}$. }
    \end{figure}
    \end{itemize}
    \end{thm}

\subsection*{Organization}
In Section \ref{s.preliminaries} we give the necessary backgrounds on the Jones polynomial, the Turaev genus, and the Kauffman 2-variable polynomial for the results of the paper. We prove a key technical result, Proposition \ref{p.tJonesDegree}, in Section \ref{ss.spanJones}. Proposition \ref{p.tJonesDegree} determines $\spn \ V_L(t)$ in Corollary \ref{c.tJonesDegree}. 
In Section \ref{ss.linkknotquestion} we include a table, Table \ref{t.parity_component}, of the classification of the number of components depending on the parity of the number of half twists in each twist region. In Section \ref{s.crossingn}, we generalize some results of Thistlethwaite to determine the minimum crossing numbers of our family of links. In particular, we prove the diagram $D(r, s, t, -u, -v)$ realizes the minimum crossing number (Theorem \ref{t.minCrossing}) using both the Kauffman 2-variable polynomial and the Jones polynomial. In Section \ref{s.ddecreasetg} we prove Theorem \ref{t.rmoveseffect}.

\section{Preliminaries} \label{s.preliminaries}

\subsection{The Jones polynomial via the Kauffman bracket state sum}
Let $D$ be a diagram of a link. The Kauffman bracket assigns a Laurent polynomial in $\mathbb{Z}[A^{\pm}]$ to $D$, denoted by $\langle A \rangle$, which is calculated recursively using the Kauffman bracket skein relations. 
\begin{itemize}
\item $\langle \vcenter{\hbox{\begin{tikzpicture}
    \draw (0,0) circle (.25cm);
\end{tikzpicture}}} \rangle = -A^2 - A^{-2}$
\item $\langle \vcenter{\hbox{\resizebox{.035\textwidth}{.035\textwidth}{\begin{tikzpicture}
\draw (-1,-1) -- (1,1);
\draw (-1,1) -- (-.15,.15);
\draw (.15,-.15) -- (1,-1);
\end{tikzpicture}}}}\rangle = A \langle \vcenter{\hbox{\resizebox{.035\textwidth}{.035\textwidth}{\begin{tikzpicture}
\foreach \ox/\cx in {-1/0, 1/0} \draw[rounded corners = 6mm] (\ox,1) -- (\cx,0) -- (\ox,-1);
\end{tikzpicture}}}} \rangle + A^{-1}\langle \vcenter{\hbox{\resizebox{.035\textwidth}{.035\textwidth}{\begin{tikzpicture}
\foreach \ox in {1,-1} \draw[rounded corners = 6mm] (-1,\ox) -- (0,0) -- (1,\ox);
\end{tikzpicture}}}} \rangle$
\end{itemize}
\begin{defn}[Kauffman state]
A Kauffman state on a link diagram $D$ is a choice of the $A$-resolution or $B$-resolution at every crossing. 
\end{defn}

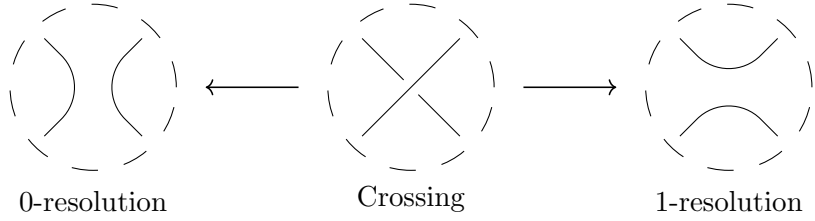
\begin{figure}[h]
\centering
\begin{tikzpicture}[scale=.65]
\draw (-1,-1) -- (1,1);
\draw (-1,1) -- (-.15,.15);
\draw (.15,-.15) -- (1,-1);
\draw (-6.5,-2.75) node[above, align=center]{$0\text{-resolution}$};
\draw (6.5,-2.75) node[above, align=center]{$1\text{-resolution}$};
\draw (0,-2.75) node[above, align=center]{Crossing};
\draw[->, line width=0.6pt] (-2.3,0) -- (-4.2,0);
\draw[->, line width=0.6pt] (2.3,0) -- (4.2,0);
\foreach \ox/\cx in {-7.5/-6.5, -5.5/-6.5} \draw[rounded corners = 6mm] (\ox,1) -- (\cx,0) -- (\ox,-1);
\foreach \ox in {1,-1} \draw[rounded corners = 6mm] (5.5,\ox) -- (6.5,0) -- (7.5,\ox);
\foreach \x in {-6.5, 0, 6.5} \draw[dash pattern=on 10pt off 10pt] (\x,0) circle (1.7cm);
\end{tikzpicture}
\caption{A Kauffman state at a crossing.}
\label{figure:resolutions}
\end{figure}

We will denote the knot diagram corresponding to the $0$-resolution by $D_0$, and we will denote the knot diagram corresponding ot the $1$-resolution by $D_\infty$.

For a crossing $x \in D$ let 
\[ \sgn_x(\sigma) = \begin{cases} 
1 & \text{if $\sigma$ chooses the $A$-resolution at $x$} \\ 
-1 & \text{if $\sigma$ chooses the $B$-resolution at $x$}
\end{cases} \]

Then define
\[ \sgn(\sigma) := \sum_{\text{$x$ a crossing in $D$}} \sgn_x(\sigma).  \]

Let $\sigma: D\rightarrow \{A, B\}$ be a Kauffman state on a link diagram $D$, and let $\sigma(D)$ be the number of circles in the crossingless diagram $D_\sigma$ obtained by replacing every crossing by the two arcs corresponding to the $A$, $B$-resolution assigned to the crossing by the Kauffman state $\sigma$. Then $|\sigma(D)|$ is the number of closed components in $D_\sigma$. The closed components of $D_\sigma$ are called the ($\sigma$)-\textit{state circles} of $D$. 

\begin{lem}{\cite[Proposition 5.1]{Lickorish}}  If $D$ is a link diagram with $n$ crossings, then Kauffman bracket of $D$ is given by
\begin{equation} \label{e.Kauffmanstatesum}
\langle D \rangle = \sum_{\text{$\sigma$ a Kauffman state}} A^{\sgn(\sigma)}(-A^2-A^{-2})^{|\sigma(D)|} \end{equation}
\end{lem}

Let $D$ be an oriented link diagram, the \textit{writhe $w(D)$} of $D$ is the sum over all crossings of $D$ of the sign of each crossing. 
\begin{figure}[H]

\centering
\begin{tikzpicture}[scale=0.9]


\draw[-{Latex[length=2mm, width=1.5mm]}, line width=0.6pt] (-4,0) -- (-2,2);
\draw[line width=0.6pt] (-2,0) -- (-2.85,0.85);
\draw[-{Latex[length=2mm, width=1.5mm]}, line width=0.6pt] (-3.15,1.15) -- (-4,2);
\draw (-3,-0.75) node[above, align=center]{$D_+$};

\draw[-{Latex[length=2mm, width=1.5mm]}, line width=0.6pt] (4,0) -- (2,2);
\draw[line width=0.6pt] (2,0) -- (2.85,0.85);
\draw[-{Latex[length=2mm, width=1.5mm  ]}, line width=0.6pt] (3.15,1.15) -- (4,2);
\draw (3,-0.75) node[above, align=center]{$D_-$};

\end{tikzpicture}

\caption{\label{f.pncrossing} $D_+$: a positive crossing. $D_-$: a negative crossing.}
\end{figure}
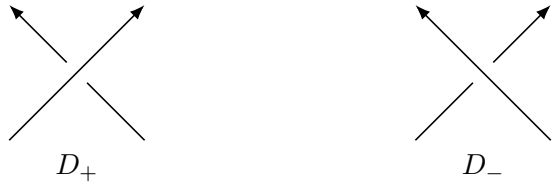
\begin{defn}
    The Jones polynomial $V_L(t)$ of an oriented link $L$ with diagram $D$ is given by
    \begin{equation} \label{e.Jonesdefn}
    V_L(t):= \left( (-A)^{-3w(D)} \langle D\rangle \right)_{t^{1/2} = A^{-2}}.  \end{equation}
\end{defn}

\begin{notate} We will denote the set of state circles from the all-$A$ state by $\sigma_A(D)$. Similarly, we will denote the set of state circles from the all-$B$ state by $\sigma_B(D)$.
\end{notate}

\begin{defn} Let $\sigma$ be a Kauffman state on a link diagram $D$. On $D_\sigma$ we put a dashed edge recording the location of the crossing as shown in Figure \ref{f.state-graph}.
\begin{figure}[H]
\centering
\def \svgwidth{.7\textwidth}
\begingroup%
  \makeatletter%
  \providecommand\color[2][]{%
    \errmessage{(Inkscape) Color is used for the text in Inkscape, but the package 'color.sty' is not loaded}%
    \renewcommand\color[2][]{}%
  }%
  \providecommand\transparent[1]{%
    \errmessage{(Inkscape) Transparency is used (non-zero) for the text in Inkscape, but the package 'transparent.sty' is not loaded}%
    \renewcommand\transparent[1]{}%
  }%
  \providecommand\rotatebox[2]{#2}%
  \newcommand*\fsize{\dimexpr\f@size pt\relax}%
  \newcommand*\lineheight[1]{\fontsize{\fsize}{#1\fsize}\selectfont}%
  \ifx\svgwidth\undefined%
    \setlength{\unitlength}{440.41377991bp}%
    \ifx\svgscale\undefined%
      \relax%
    \else%
      \setlength{\unitlength}{\unitlength * \real{\svgscale}}%
    \fi%
  \else%
    \setlength{\unitlength}{\svgwidth}%
  \fi%
  \global\let\svgwidth\undefined%
  \global\let\svgscale\undefined%
  \makeatother%
  \begin{picture}(1,0.18338445)%
    \lineheight{1}%
    \setlength\tabcolsep{0pt}%
    \put(0,0){\includegraphics[width=\unitlength,page=1]{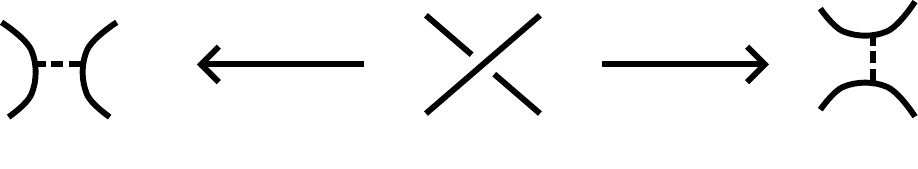}}%
    \put(0.03488482,0.00360215){\color[rgb]{0,0,0}\makebox(0,0)[lt]{\lineheight{1.25}\smash{\begin{tabular}[t]{l}$A$\end{tabular}}}}%
    \put(0.93333124,0.00360215){\color[rgb]{0,0,0}\makebox(0,0)[lt]{\lineheight{1.25}\smash{\begin{tabular}[t]{l}$B$\end{tabular}}}}%
  \end{picture}%
\endgroup%

\caption{\label{f.state-graph} $A$-and $B$-resolutions and the dashed edges recording the location of the crossing. }
\end{figure}

The \textit{$\sigma$-state graph} is the graph where the vertices are the state circles and the edges are the same as the edges recording the location of the crossing. A diagram is called \textit{$A$-adequate} if every edge connects two distinct vertices in the state graph coming from the all $A$-state $\sigma_A$. A diagram is called \textit{$B$-adequate} if every edge connects two distinct vertices in the state graph coming from the all $B$-state $\sigma_B$. 
\end{defn}

\begin{defn} \label{d.adequate}
    A link is \textit{$A$-adequate} if it admits an $A$-adequate diagram. A link is I if it admits a $B$-adequate diagram. A link is \textit{adequate} if it admits a diagram that is both $A$-and $B$-adequate. 
\end{defn}
    Every alternating link is adequate. 

\subsection{The Turaev genus}
\begin{defn}[Turaev genus] \label{d.tg}
The \textit{Turaev genus} of a link diagram $D$ is the genus of a closed, oriented surface, called the Turaev surface, in $S^3$ obtained from the diagram $D$: 
\begin{equation} \label{e.tgdefn0} g_T(D) = \frac{c(D) + 2 - |\sigma_A(D)| - |\sigma_B(D)|}{2}.\end{equation}
The\textit{ Turaev genus} of a link $L$ is the minimum genus over all Turaev surfaces of $L$. 
\begin{equation} \label{e.tgdefn}
g_T(L):= \min \{ g_T(D) \ | \ \text{$D$ is a diagram for $L$} \}. \end{equation}
\end{defn}

On a Turaev surface that realizes the Turaev genus, the knot is alternating, and therefore this invariant is also related to the dealternating genus of a knot. See \cite{CK-surveyTG} for an excellent survey on the Turaev genus.

\subsection{The degree span of the Jones polynomial} 
\begin{defn}[Jones polynomial span] \label{d.Mdmd}
Define 
\begin{align} \label{e.M(D)} M(D) &= c(D) + 2|\sigma_A(D)|    \\
\label{e.m(D)} m(D) &= -c(D) - 2|\sigma_B(D)|   \end{align}
\end{defn}
It is well known that the maximum and minimum degrees of the Kauffman bracket  in $A$, denoted by $\deg_M\langle D \rangle,  \deg_m\langle D \rangle$, respectively, are bounded by $M(D)$ and $m(D)$ \cite[Lemma 5.4]{Lickorish}:
\begin{align*}
&\deg_M\langle D \rangle \leq M(D) \\ 
&\deg_m\langle D \rangle \geq m(D)
\end{align*}
If $D$ is an adequate diagram, then equalities are achieved. 

We now consider the conversion from the degree of the Kauffman bracket to the degree of the Jones polynomial. From Eq. \eqref{e.Jonesdefn}, the variable substitution $t^{1/2} = A^{-2}$ means that we multiply $M(D)$ by $-1/4$ to obtain the lowest $t$-degree term in $V_L(t)$. We do the same to $m(D)$ for the highest degree term. Then we adjust by the writhe by multiplying by the monomial $(-1)^{3w(D)}t^{3w(D)/4}$. Define 
\begin{align} \label{e.writhe_adjust}
& M_J(D) := -\frac{m(D)}{4} + \frac{3w(D)}{4}\\ 
& m_J(D) := -\frac{M(D)}{4} + \frac{3w(D)}{4}  
\end{align}
These bounds come from the extreme Kauffman states: $M_J(D)$ comes from the all-$B$ Kauffman state which chooses the $B$-resolution at every crossing of $D$, and $m_J(D)$ comes from the all-$A$ Kauffman state which chooses the $A$-resolution at every crossing of $D$. 
Write 
\[ J_K(t) = a_m t^{m(L)} + \cdots + a_M t^{M(L)},  \]
where $a_m$ and $a_M$ are both nonzero. The \textit{span} of the Jones polynomial of a link $L$ is defined as
\begin{equation} \label{e.spandefn} \spn(V_L(t)) = M(L) - m(L). \end{equation}
It is well known \cite{BM-spreadofJones, DFKLS} that 
\begin{equation} \label{e.boundspan} \spn(V_L(t)) \leq c(L) - g_T(L), \end{equation}
with equality if $L$ admits an adequate diagram \cite{Abe-tg, Thislethwaite-adequate}.

\subsection{The Kauffman 2-variable polynomial} 
The Kauffman polynomial \cite{Kauffman_polynomial} is a two-variable polynomial in $\mathbb{Z}[a^{\pm 1}, z^{\pm 1}]$ defined for an oriented link by 
\begin{equation} \label{e.Kpolydefn} F(L) = a^{-w(D)} \Lambda(D). \end{equation}
Here $\Lambda(D)$ is a two-variable polynomial that is invariant with respect to all the Reidemeister moves except Type I. We will follow Lickorish's description of the Kauffman two-variable polynomial. 

\begin{thm}{\cite[Theorem 15.5]{Lickorish}}
There exists a function 
\[ \Lambda: \{\text{Unoriented link diagrams in $S^2$} \} \rightarrow  \mathbb{Z}[a^{\pm 1}, z^{\pm 1}]\]
that is defined uniquely by the following: 
\begin{enumerate}[(i)]
    \item $\Lambda(U) = 1$, where $U$ is the zero-crossing diagram of the unknot; 
    \item $\Lambda(D)$ is unchanged by Reidemeister moves of Types II and III on the diagram $D$; 
    \item $\Lambda(D)$ changes with respect to Reidemeister move of Type I as follows: 
    \[ \Lambda(\vcenter{\hbox{\includegraphics[scale=.55]{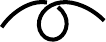}}}) = a^{-1}\Lambda (\vcenter{\hbox{\includegraphics[scale=.55]{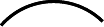}}});\]
    \item If $D_+, D_-, D_0$, and $D_{\infty}$ are four diagrams exactly the same except in a small disk $D^2$ intersecting the diagram at four points where $D_+\cap D^2$ is the positive crossing, $D_-\cap D^2$ is the negative crossing, $D_0 \cap D^2$ is the $A$-resolution, and $D_{\infty} \cap D^2$ is the $B$-resolution, respectively.  Then 
    \begin{equation} \label{e.lambda} \Lambda (D_+) + \Lambda (D_-) = z(\Lambda(D_0) + \Lambda(D_\infty)). \end{equation}
\end{enumerate}
\end{thm}

Let $D$ be the diagram of a knot, a \textit{bridge} is a subarc of $D$ that contains an overcrossing. The length of a bridge is the number of overcrossings it contains. 
\begin{figure}[H]
\includegraphics[scale=.5]{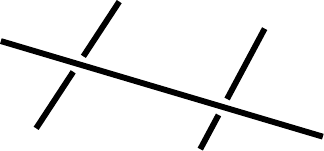}
\caption{\label{f.bridge2} A bridge of length 2.}
\end{figure}

Note that an alternating link diagram only has bridges of length 1. Using the Kauffman polynomial, Thistlethwaite showed the following. 

\begin{thm}\cite[Theorem 4]{Thistlethwaite-Kauffman} \label{t.T-K}
    Let $D$ be an $n$-crossing link diagram which is a connected sum of link diagrams $D_1, \ldots, D_k$. Let $\Lambda_D(a, z) = \sum u_{rs}a^rz^s$, and let $b_1, \ldots, b_k$ be the lengths of the longest bridges of $D_1, \ldots, D_k$, respectively. Then, for each non-zero coefficient $u_{rs}$, $|r|+s \leq n$ and $s\leq n - (b_1+\cdots +b_k)$. 
\end{thm}

In the same paper Thisthlethwaite makes the following observation regarding the Kauffman polynomial for a prime, reduced, alternating link. 
\begin{obs}\cite[Page 315]{Thistlethwaite-Kauffman} \label{o.T-K}
Let $L$ be a link admitting a prime, reduced alternating diagram $D$ with $n$ crossings and writhe $w$. The coefficient of $z^{n-1}$ in $\Lambda_D(a, z)$ is $k(a^{-1}+a)$ with $k>0$.     
\end{obs}

In particular, a prime alternating link realizes the maxmimum possible $z$-degree of the Kauffman polynomial. 

In \cite{Lee-nearalt} the third author introduced \textit{near-alternating} knots, which includes pretzel links with one negative twist region. 
\begin{lem}{\cite[Lemma 5.1]{Lee-nearalt}} \label{l.nearalt}
Let $D$ be a near-alternating diagram of a link with $n$ crossings and a maximal negative twist region of weight $r\leq 0$ with $|r|\geq 2$, then the $z$-degree of $\Lambda(D)$ is equal to $n-2$. 
\end{lem}
In Section \ref{s.crossingn}, we will use Lemma \ref{l.nearalt} to determine the crossing numbers of our family of links $L = D(r, s, t, -u, -v)$. 

\begin{defn}[Connected sum of links]
Two oriented knots or links $L_1, L_2$ with diagrams $D_1, D_2$ can be summed to a new knot or link $L_1 + L_2 = D_1 + D_2$ by placing the diagrams $D_1, D_2$ side by side and joining them so that the orientation is preserved in the sum. The result is called a \textit{connected sum} of $L_1$ and $L_2$. 
\end{defn}
\begin{rmk}
    If $L_1, L_2$ are both knots, then the isotopy type of the connected sum $L_1 + L_2$ does not depend on the choice of the diagrams and how they are joined. 
\end{rmk}

\begin{prop}{\cite[Portion of Proposition 16.2]{Lickorish}}
If $L_1$ and $L_2$ are oriented links, then 
\[ F(L_1 + L_2) = F(L_1) F(L_2). \]
\end{prop}


\section{Extreme coefficients of the Jones polynomial.}
In this section, we study the degree of the Jones polynomial of the family of links $L = D(r, s, t, -u, -v)$, with $r, s, t, u, v \geq 2$. Through Lemmas \ref{l.vdegree} and \ref{l.vdegreemirror}, we tease out conditions on $r, s, t, u, v$ that will allow us to determine the maximum degree and minium degree, and the leading and trailing coefficients of the Jones polynomial $V_L(t)$ to be 2. This will allow us to compute the Turaev genus of $L$, and the defect $\delta(L)$ of the span of the Jones polynomial for the link with respect to the upper bound $c(L) - g_T(L)$. 

We need a well-known fact about how the Kauffman bracket polynomial changes with respect to a Type I Reidemeister move. 
\begin{lem} \label{l.r1_lowers_degree}
    Suppose we have a link $L$ represented by link diagrams $D_1, D_2$ such that $D_1$ and $D_2$ agree everywhere except inside a small disk intersecting each diagram transversely in four points. Inside of the small disk, they differ by a Type I Reidmeister move as shown in the following figure: 
    
    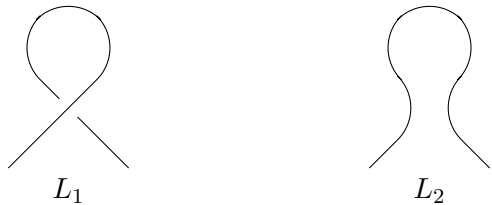
\begin{figure}[h]
    \centering
    \begin{tikzpicture}[scale=0.8]
    
    \draw[rounded corners =6mm] (-4,0) -- (-2,2) -- (-3,3) -- (-4,2) -- (-3.15,1.15);
    \draw (-2.85,0.85) -- (-2,0);
    \draw (-3,-0.75) node[above, align=center]{$L_1$};
    
    \draw[rounded corners = 6mm] (2,0) -- (3,1) -- (2,2) -- (3,3) -- (4,2) -- (3,1) -- (4,0);
    \draw (3,-0.75) node[above, align=center]{$L_2$};
    
    \end{tikzpicture}
    \caption{$D_1$ and $D_2$.}
    \label{figure:link_diagram}
    \end{figure}
    
    Recall $M(D)$ (Definition \ref{d.Mdmd}, Eq. \eqref{e.M(D)}) for a link diagram $D$ is a diagrammatic upper bound of the degree of the Kauffman bracket. Then \[ M(D_1)  = M(D_2) - 4.\] 
    If instead $D, D'$ defer by the other Reidemeister I move that is the mirror image, then we have
    Then \[ M(D_1)  = M(D_2) +4.\]
\end{lem}
\begin{proof}
This follows from an elementary calculation, which can be found in \cite{Lickorish}, of how $M(D)$ changes under the Type I Reidemeister move.
\end{proof}

\begin{defn}[Degree of a Kauffman state]
The (maximum) \textit{degree of a Kauffman state $\sigma$} is the degree of the corresponding monomial in the state sum formula \eqref{e.Kauffmanstatesum} of the Jones polynomial. 
\[ \deg \ \sigma := \sgn(\sigma) + 2|\sigma|.\]
\end{defn} 

\subsection{Maximum degree}
We distinguish three Kauffman states and label them $\sigma_1, \sigma_2, \sigma_3$. See Figure \ref{f.three_Kauffman_states}.

\begin{figure}[H]
     \centering
     \begin{subfigure}[b]{0.3\textwidth}
         \centering
         \def \svgwidth{\textwidth}
\begingroup%
  \makeatletter%
  \providecommand\color[2][]{%
    \errmessage{(Inkscape) Color is used for the text in Inkscape, but the package 'color.sty' is not loaded}%
    \renewcommand\color[2][]{}%
  }%
  \providecommand\transparent[1]{%
    \errmessage{(Inkscape) Transparency is used (non-zero) for the text in Inkscape, but the package 'transparent.sty' is not loaded}%
    \renewcommand\transparent[1]{}%
  }%
  \providecommand\rotatebox[2]{#2}%
  \newcommand*\fsize{\dimexpr\f@size pt\relax}%
  \newcommand*\lineheight[1]{\fontsize{\fsize}{#1\fsize}\selectfont}%
  \ifx\svgwidth\undefined%
    \setlength{\unitlength}{548.26098344bp}%
    \ifx\svgscale\undefined%
      \relax%
    \else%
      \setlength{\unitlength}{\unitlength * \real{\svgscale}}%
    \fi%
  \else%
    \setlength{\unitlength}{\svgwidth}%
  \fi%
  \global\let\svgwidth\undefined%
  \global\let\svgscale\undefined%
  \makeatother%
  \begin{picture}(1,0.67181646)%
    \lineheight{1}%
    \setlength\tabcolsep{0pt}%
    \put(0,0){\includegraphics[width=\unitlength,page=1]{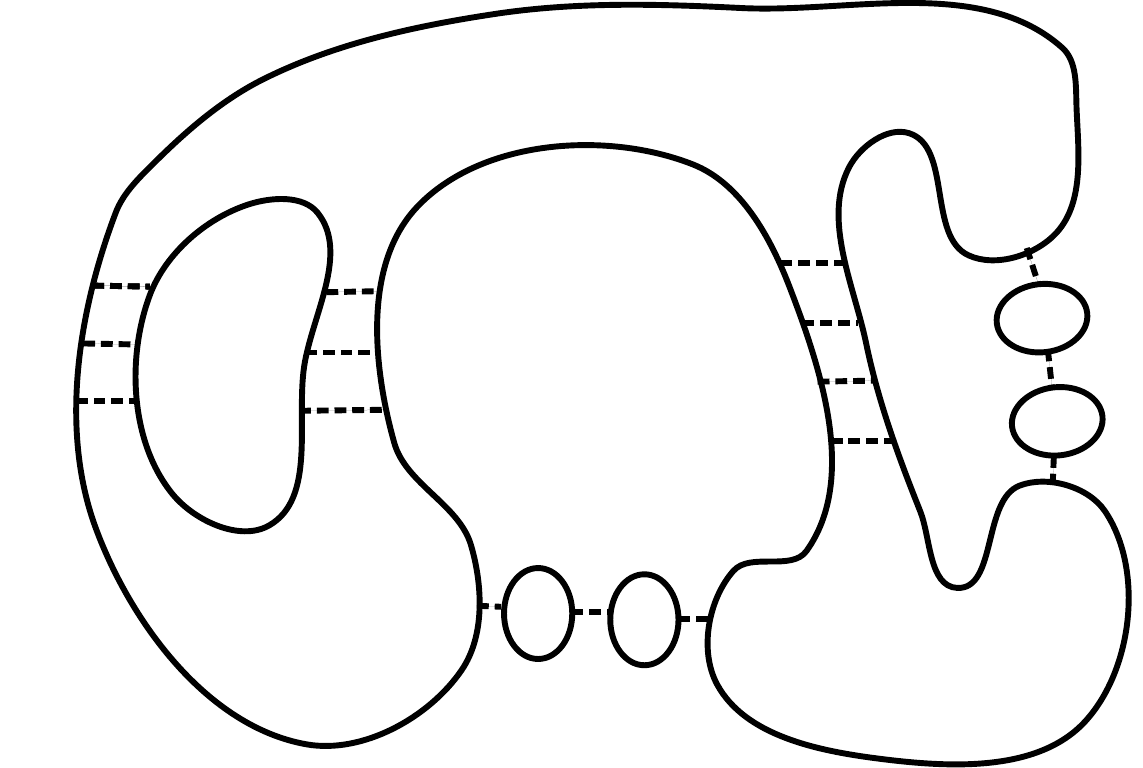}}%
    \put(-0.00124239,0.38167522){\color[rgb]{0,0,0}\makebox(0,0)[lt]{\lineheight{1.25}\smash{\begin{tabular}[t]{l}$r$\end{tabular}}}}%
    \put(0.35142162,0.36676181){\color[rgb]{0,0,0}\makebox(0,0)[lt]{\lineheight{1.25}\smash{\begin{tabular}[t]{l}$s$\end{tabular}}}}%
    \put(0.66376078,0.36116207){\color[rgb]{0,0,0}\makebox(0,0)[lt]{\lineheight{1.25}\smash{\begin{tabular}[t]{l}$t$\end{tabular}}}}%
    \put(0.96987811,0.36613959){\color[rgb]{0,0,0}\makebox(0,0)[lt]{\lineheight{1.25}\smash{\begin{tabular}[t]{l}$u$\end{tabular}}}}%
    \put(0.50136933,0.04695637){\color[rgb]{0,0,0}\makebox(0,0)[lt]{\lineheight{1.25}\smash{\begin{tabular}[t]{l}$v$\end{tabular}}}}%
  \end{picture}%
\endgroup%

         \caption{$\sigma_1$}
         \label{f.sigma1}
     \end{subfigure}
     \hfill
     \begin{subfigure}[b]{0.3\textwidth}
         \centering
          \def \svgwidth{\textwidth}
\begingroup%
  \makeatletter%
  \providecommand\color[2][]{%
    \errmessage{(Inkscape) Color is used for the text in Inkscape, but the package 'color.sty' is not loaded}%
    \renewcommand\color[2][]{}%
  }%
  \providecommand\transparent[1]{%
    \errmessage{(Inkscape) Transparency is used (non-zero) for the text in Inkscape, but the package 'transparent.sty' is not loaded}%
    \renewcommand\transparent[1]{}%
  }%
  \providecommand\rotatebox[2]{#2}%
  \newcommand*\fsize{\dimexpr\f@size pt\relax}%
  \newcommand*\lineheight[1]{\fontsize{\fsize}{#1\fsize}\selectfont}%
  \ifx\svgwidth\undefined%
    \setlength{\unitlength}{630.34741019bp}%
    \ifx\svgscale\undefined%
      \relax%
    \else%
      \setlength{\unitlength}{\unitlength * \real{\svgscale}}%
    \fi%
  \else%
    \setlength{\unitlength}{\svgwidth}%
  \fi%
  \global\let\svgwidth\undefined%
  \global\let\svgscale\undefined%
  \makeatother%
  \begin{picture}(1,0.60403847)%
    \lineheight{1}%
    \setlength\tabcolsep{0pt}%
    \put(0,0){\includegraphics[width=\unitlength,page=1]{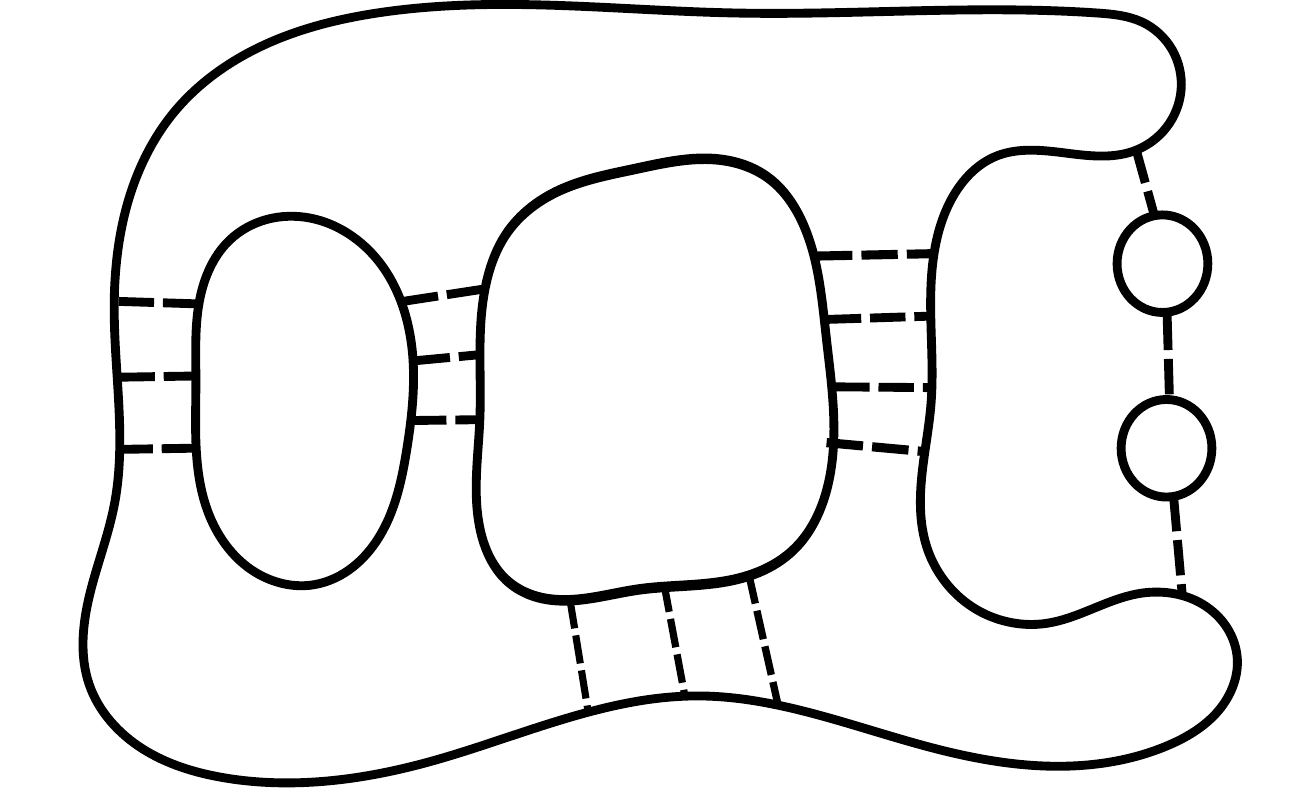}}%
    \put(-0.0028816,0.32462824){\color[rgb]{0,0,0}\makebox(0,0)[lt]{\lineheight{1.25}\smash{\begin{tabular}[t]{l}$r$\end{tabular}}}}%
    \put(0.23555607,0.31253354){\color[rgb]{0,0,0}\makebox(0,0)[lt]{\lineheight{1.25}\smash{\begin{tabular}[t]{l}$s$\\\end{tabular}}}}%
    \put(0.55174512,0.34708969){\color[rgb]{0,0,0}\makebox(0,0)[lt]{\lineheight{1.25}\smash{\begin{tabular}[t]{l}$t$\end{tabular}}}}%
    \put(0.93013523,0.32290043){\color[rgb]{0,0,0}\makebox(0,0)[lt]{\lineheight{1.25}\smash{\begin{tabular}[t]{l}$u$\end{tabular}}}}%
    \put(0.48090496,0.00671131){\color[rgb]{0,0,0}\makebox(0,0)[lt]{\lineheight{1.25}\smash{\begin{tabular}[t]{l}$v$\\\end{tabular}}}}%
  \end{picture}%
\endgroup%

         \caption{$\sigma_2$}
         \label{f.sigma2}
     \end{subfigure}
     \hfill
     \begin{subfigure}[b]{0.3\textwidth}
         \centering
         \def \svgwidth{\textwidth}
\begingroup%
  \makeatletter%
  \providecommand\color[2][]{%
    \errmessage{(Inkscape) Color is used for the text in Inkscape, but the package 'color.sty' is not loaded}%
    \renewcommand\color[2][]{}%
  }%
  \providecommand\transparent[1]{%
    \errmessage{(Inkscape) Transparency is used (non-zero) for the text in Inkscape, but the package 'transparent.sty' is not loaded}%
    \renewcommand\transparent[1]{}%
  }%
  \providecommand\rotatebox[2]{#2}%
  \newcommand*\fsize{\dimexpr\f@size pt\relax}%
  \newcommand*\lineheight[1]{\fontsize{\fsize}{#1\fsize}\selectfont}%
  \ifx\svgwidth\undefined%
    \setlength{\unitlength}{626.74692229bp}%
    \ifx\svgscale\undefined%
      \relax%
    \else%
      \setlength{\unitlength}{\unitlength * \real{\svgscale}}%
    \fi%
  \else%
    \setlength{\unitlength}{\svgwidth}%
  \fi%
  \global\let\svgwidth\undefined%
  \global\let\svgscale\undefined%
  \makeatother%
  \begin{picture}(1,0.61105746)%
    \lineheight{1}%
    \setlength\tabcolsep{0pt}%
    \put(0,0){\includegraphics[width=\unitlength,page=1]{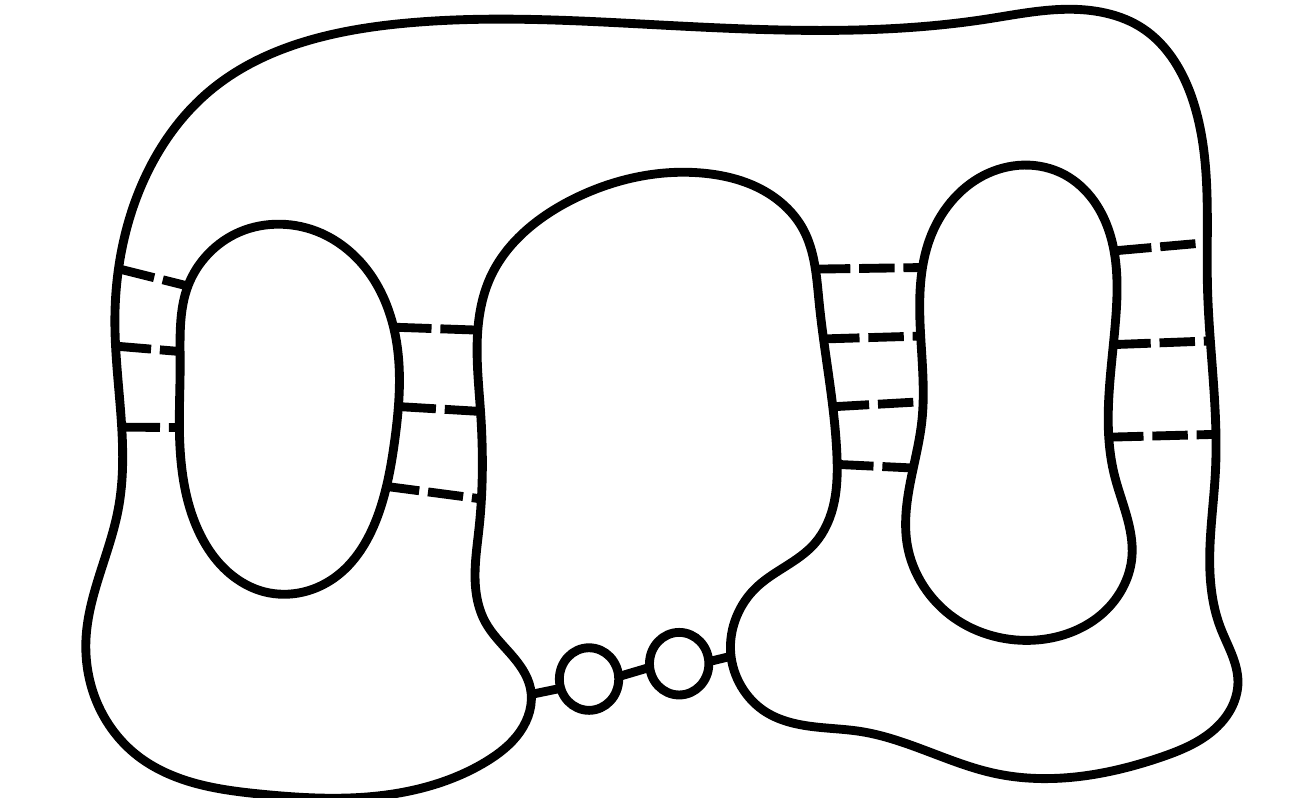}}%
    \put(0.01061825,0.31094431){\color[rgb]{0,0,0}\makebox(0,0)[lt]{\lineheight{1.25}\smash{\begin{tabular}[t]{l}$r$\end{tabular}}}}%
    \put(0.21950816,0.28882655){\color[rgb]{0,0,0}\makebox(0,0)[lt]{\lineheight{1.25}\smash{\begin{tabular}[t]{l}$s$\end{tabular}}}}%
    \put(0.56724845,0.29497034){\color[rgb]{0,0,0}\makebox(0,0)[lt]{\lineheight{1.25}\smash{\begin{tabular}[t]{l}$t$\end{tabular}}}}%
    \put(0.94325029,0.31094431){\color[rgb]{0,0,0}\makebox(0,0)[lt]{\lineheight{1.25}\smash{\begin{tabular}[t]{l}$u$\end{tabular}}}}%
    \put(0.46034595,0.00498208){\color[rgb]{0,0,0}\makebox(0,0)[lt]{\lineheight{1.25}\smash{\begin{tabular}[t]{l}$v$\\\end{tabular}}}}%
  \end{picture}%
\endgroup%

         \caption{$\sigma_3$}
         \label{f.sigma_3}
     \end{subfigure}
        \caption{Kauffman states $\sigma_1, \sigma_2, \sigma_3$.}
        \label{f.three_Kauffman_states}
\end{figure}
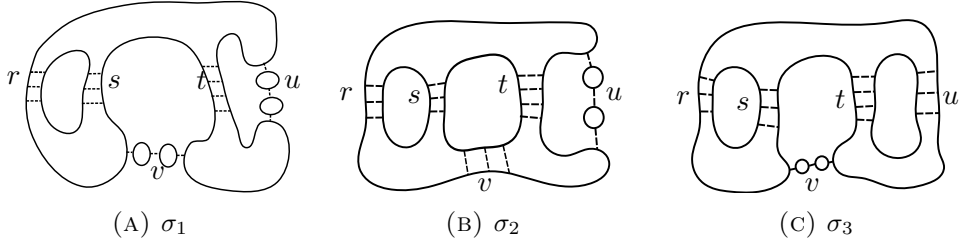

Let $\sigma$ be a Kauffman state on a link diagram $D$, we denote by $\langle S_\sigma \rangle$ the term in the state sum Eq. \eqref{e.Kauffmanstatesum} corresponding to $\sigma$: 
\[ \langle S_\sigma \rangle := A^{\sgn(\sigma)} (-A^2 - A^{-2})^{|\sigma(D)|}. \]
The state sum Eq. \eqref{e.Kauffmanstatesum} can now be written as 
\begin{equation} \label{e.nstatesum} \langle D \rangle = \sum{\sigma \in S} \langle \sigma \rangle. \end{equation}
To determine the degree of the Jones polynomial, which the Kauffman bracket of $D$ determines up to a shift, we index a Kauffman state $\sigma\in S$ with vector $v(\sigma) = (\sigma_r, \sigma_s, \sigma_u, \sigma_v)$ with $0 \leq \sigma_r \leq r$ is the number of crossings on which $\sigma$ chooses the $1$-resolution on the twist region $T_r$, 
$0 \leq \sigma_s \leq s, 0 \leq \sigma_u \leq u$,  \ldots, and so on.

\begin{lem}  \label{l.vdegree}
Let $L =  D(r, s, t, u, v)$. We have
    \begin{equation} \label{e.Dstatesum} \langle D(r, s, t, u, v) \rangle = S_1 + S_2 + S_3 + S_4, \end{equation}
    where 
    \[ S_1 = \sum_{\sigma_t = 0, \ \sigma_u \not= u, \ \sigma_v \not= v} \langle \sigma \rangle,  \]
    \[ S_2  = \sum_{\sigma_t = 0, \ \sigma_u = u, \ \sigma_v \not= v} \langle \sigma \rangle, \]
    \[ S_3 = \sum_{\sigma_t = 0, \ \sigma_v = v} \langle \sigma \rangle, \]
    and 
    \[ S_4 =\sum_{\sigma_t>0} \langle \sigma \rangle.  \]
    Furthermore, 
    \begin{align*}
        \deg \ S_1 &= \deg \ \sigma_1  - 4t\\ 
        &= (2r+t+6u) - 4t = 2r-3t+6u \\ 
        \deg \ S_2 &= \deg \ \sigma_2 \\
        &= 2r+t+2u+4 = \deg \ \sigma_1 - (4u-4) \\ 
        \deg \ S_3 &= \deg \ \sigma_3 = \deg \ \sigma_2   \\
        \deg \ S_4 &= \deg \ \sigma_1 - 4(t-1)\\
    \end{align*}
   
\end{lem}

\begin{proof}
We separate the terms in the state sum \eqref{e.nstatesum} by $\sigma_t=0, \sigma_t > 0$: 
\[ \langle D \rangle = \sum_{\sigma_t = 0} \langle \sigma \rangle + \sum_{\sigma_t > 0} \langle \sigma \rangle.   \]
When $\sigma_t = 0$, the Kauffman state $\sigma$ chooses the $A$-resolution on every crossing in the twist region $T_t$. We further decompose into two cases 
\begin{enumerate} 
\item $\sigma_u \not= u$ and $\sigma_v \not= v$.  \\  
The $\sigma$-state graph has $t$ one-edged loops. See Figure \ref{f.tequals0} for an illustration. Therefore by Lemma \ref{l.r1_lowers_degree}, denote the diagram obtained by performing $t$ Reidemeister I moves to untwist the $t$ loops, 
$$\deg \sum_{\text{$\sigma_t = 0, \sigma_u \not= u$ and $\sigma_v \not= v$}} \langle \sigma \rangle = \deg \ \sigma_1  - 4t$$. 
\begin{figure}[H]
\centering
\def\svgwidth{.3\textwidth}
\begingroup%
  \makeatletter%
  \providecommand\color[2][]{%
    \errmessage{(Inkscape) Color is used for the text in Inkscape, but the package 'color.sty' is not loaded}%
    \renewcommand\color[2][]{}%
  }%
  \providecommand\transparent[1]{%
    \errmessage{(Inkscape) Transparency is used (non-zero) for the text in Inkscape, but the package 'transparent.sty' is not loaded}%
    \renewcommand\transparent[1]{}%
  }%
  \providecommand\rotatebox[2]{#2}%
  \newcommand*\fsize{\dimexpr\f@size pt\relax}%
  \newcommand*\lineheight[1]{\fontsize{\fsize}{#1\fsize}\selectfont}%
  \ifx\svgwidth\undefined%
    \setlength{\unitlength}{482.83043569bp}%
    \ifx\svgscale\undefined%
      \relax%
    \else%
      \setlength{\unitlength}{\unitlength * \real{\svgscale}}%
    \fi%
  \else%
    \setlength{\unitlength}{\svgwidth}%
  \fi%
  \global\let\svgwidth\undefined%
  \global\let\svgscale\undefined%
  \makeatother%
  \begin{picture}(1,0.55927966)%
    \lineheight{1}%
    \setlength\tabcolsep{0pt}%
    \put(0,0){\includegraphics[width=\unitlength,page=1]{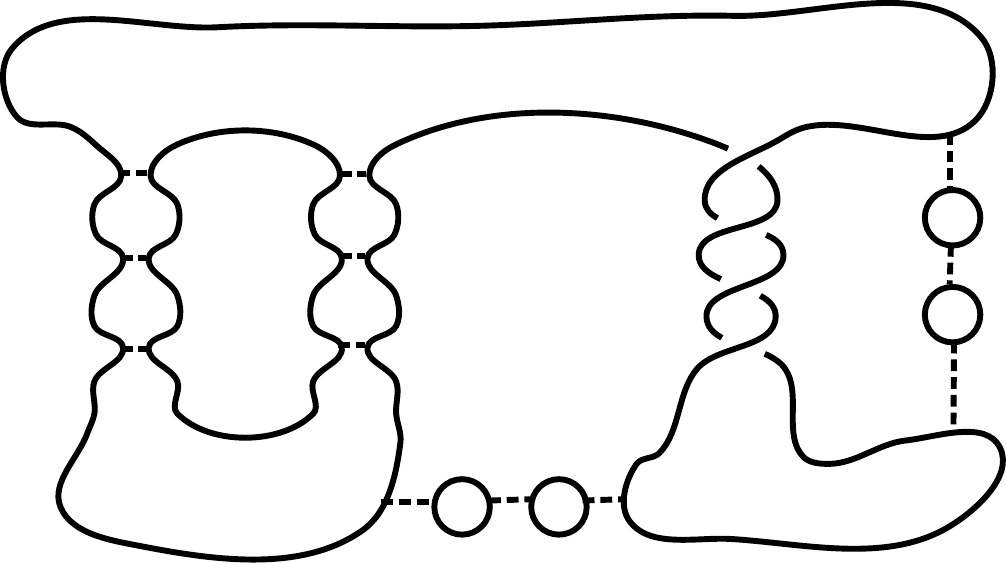}}%
    \put(0.03271092,0.28375298){\color[rgb]{0,0,0}\makebox(0,0)[lt]{\lineheight{1.25}\smash{\begin{tabular}[t]{l}$r$\end{tabular}}}}%
    \put(0.26530737,0.28123674){\color[rgb]{0,0,0}\makebox(0,0)[lt]{\lineheight{1.25}\smash{\begin{tabular}[t]{l}$s$\end{tabular}}}}%
    \put(0.64584783,0.28123674){\color[rgb]{0,0,0}\makebox(0,0)[lt]{\lineheight{1.25}\smash{\begin{tabular}[t]{l}$t$\end{tabular}}}}%
    \put(0.87962512,0.27872049){\color[rgb]{0,0,0}\makebox(0,0)[lt]{\lineheight{1.25}\smash{\begin{tabular}[t]{l}$u$\end{tabular}}}}%
    \put(0.48940582,0.105101){\color[rgb]{0,0,0}\makebox(0,0)[lt]{\lineheight{1.25}\smash{\begin{tabular}[t]{l}$v$\end{tabular}}}}%
  \end{picture}%
\endgroup%

\caption{\label{f.tequals0} There are $t$ loops corresponding to $t$ crossings. }
\end{figure}
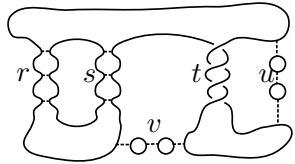
\item $\sigma_u = u$. \\
In this case, the edges in the $\sigma$-state graph corresponding to the twist region $T_t$ connect a distinct closed component, regardless of the choice of resolution of $\sigma$ for the rest of the crossings. See Figure \ref{f.tequals0uequalsu}. Therefore, 
$$\deg \sum_{\sigma_t = 0, \sigma_u = u} \langle \sigma \rangle = \deg \ \sigma_2$$.
\begin{figure}[H]
\centering
\def \svgwidth{.3\textwidth}
\begingroup%
  \makeatletter%
  \providecommand\color[2][]{%
    \errmessage{(Inkscape) Color is used for the text in Inkscape, but the package 'color.sty' is not loaded}%
    \renewcommand\color[2][]{}%
  }%
  \providecommand\transparent[1]{%
    \errmessage{(Inkscape) Transparency is used (non-zero) for the text in Inkscape, but the package 'transparent.sty' is not loaded}%
    \renewcommand\transparent[1]{}%
  }%
  \providecommand\rotatebox[2]{#2}%
  \newcommand*\fsize{\dimexpr\f@size pt\relax}%
  \newcommand*\lineheight[1]{\fontsize{\fsize}{#1\fsize}\selectfont}%
  \ifx\svgwidth\undefined%
    \setlength{\unitlength}{482.83043569bp}%
    \ifx\svgscale\undefined%
      \relax%
    \else%
      \setlength{\unitlength}{\unitlength * \real{\svgscale}}%
    \fi%
  \else%
    \setlength{\unitlength}{\svgwidth}%
  \fi%
  \global\let\svgwidth\undefined%
  \global\let\svgscale\undefined%
  \makeatother%
  \begin{picture}(1,0.55927966)%
    \lineheight{1}%
    \setlength\tabcolsep{0pt}%
    \put(0,0){\includegraphics[width=\unitlength,page=1]{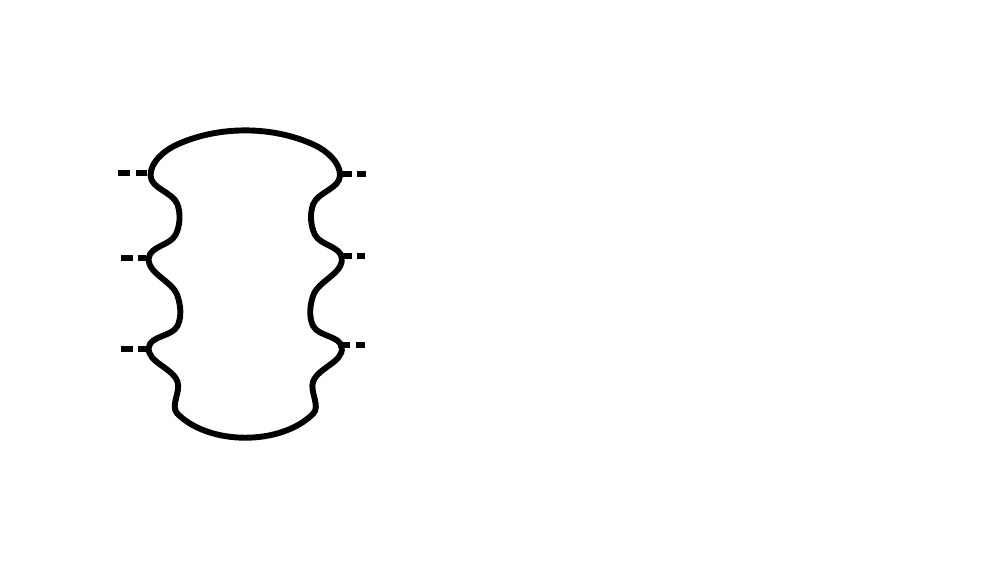}}%
    \put(0.03271092,0.28375298){\color[rgb]{0,0,0}\makebox(0,0)[lt]{\lineheight{1.25}\smash{\begin{tabular}[t]{l}$r$\end{tabular}}}}%
    \put(0.26530737,0.28123674){\color[rgb]{0,0,0}\makebox(0,0)[lt]{\lineheight{1.25}\smash{\begin{tabular}[t]{l}$s$\end{tabular}}}}%
    \put(0.64584783,0.28123674){\color[rgb]{0,0,0}\makebox(0,0)[lt]{\lineheight{1.25}\smash{\begin{tabular}[t]{l}$t$\end{tabular}}}}%
    \put(0.87962512,0.27872049){\color[rgb]{0,0,0}\makebox(0,0)[lt]{\lineheight{1.25}\smash{\begin{tabular}[t]{l}$u$\end{tabular}}}}%
    \put(0.48940582,0.105101){\color[rgb]{0,0,0}\makebox(0,0)[lt]{\lineheight{1.25}\smash{\begin{tabular}[t]{l}$v$\end{tabular}}}}%
    \put(0,0){\includegraphics[width=\unitlength,page=2]{te0sueusvev.pdf}}%
  \end{picture}%
\endgroup%

\caption{\label{f.tequals0uequalsu} The $t$ edges no longer connect the same component in the $\sigma$-state graph.  }
\end{figure}
\item $\sigma_v = v$. \\
This case is similar to the previous that the edges in the $\sigma$-state graph corresponding to the twist region $T_t$ connect a distinct closed component, regardless of the choice of resolution of $\sigma$ for the rest of the crossings. See Figure \ref{f.tequals0vequalsv}. 
$$\deg \sum_{\sigma_t = 0, \ \sigma_v = v} \langle \sigma \rangle = \deg \ \sigma_3 $$
\begin{figure}[H]
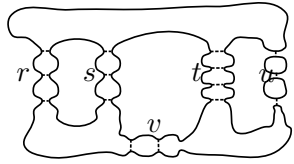

\centering
\def \svgwidth{.3\textwidth}
\begingroup%
  \makeatletter%
  \providecommand\color[2][]{%
    \errmessage{(Inkscape) Color is used for the text in Inkscape, but the package 'color.sty' is not loaded}%
    \renewcommand\color[2][]{}%
  }%
  \providecommand\transparent[1]{%
    \errmessage{(Inkscape) Transparency is used (non-zero) for the text in Inkscape, but the package 'transparent.sty' is not loaded}%
    \renewcommand\transparent[1]{}%
  }%
  \providecommand\rotatebox[2]{#2}%
  \newcommand*\fsize{\dimexpr\f@size pt\relax}%
  \newcommand*\lineheight[1]{\fontsize{\fsize}{#1\fsize}\selectfont}%
  \ifx\svgwidth\undefined%
    \setlength{\unitlength}{482.83043569bp}%
    \ifx\svgscale\undefined%
      \relax%
    \else%
      \setlength{\unitlength}{\unitlength * \real{\svgscale}}%
    \fi%
  \else%
    \setlength{\unitlength}{\svgwidth}%
  \fi%
  \global\let\svgwidth\undefined%
  \global\let\svgscale\undefined%
  \makeatother%
  \begin{picture}(1,0.55927966)%
    \lineheight{1}%
    \setlength\tabcolsep{0pt}%
    \put(0,0){\includegraphics[width=\unitlength,page=1]{te0suvev.pdf}}%
    \put(0.03271092,0.28375298){\color[rgb]{0,0,0}\makebox(0,0)[lt]{\lineheight{1.25}\smash{\begin{tabular}[t]{l}$r$\end{tabular}}}}%
    \put(0.26530737,0.28123674){\color[rgb]{0,0,0}\makebox(0,0)[lt]{\lineheight{1.25}\smash{\begin{tabular}[t]{l}$s$\end{tabular}}}}%
    \put(0.64584783,0.28123674){\color[rgb]{0,0,0}\makebox(0,0)[lt]{\lineheight{1.25}\smash{\begin{tabular}[t]{l}$t$\end{tabular}}}}%
    \put(0.87962512,0.27872049){\color[rgb]{0,0,0}\makebox(0,0)[lt]{\lineheight{1.25}\smash{\begin{tabular}[t]{l}$u$\end{tabular}}}}%
    \put(0.48940582,0.105101){\color[rgb]{0,0,0}\makebox(0,0)[lt]{\lineheight{1.25}\smash{\begin{tabular}[t]{l}$v$\end{tabular}}}}%
    \put(0,0){\includegraphics[width=\unitlength,page=2]{te0suvev.pdf}}%
  \end{picture}%
\endgroup%

\caption{\label{f.tequals0vequalsv} The case $\sigma_v = v$ but $\sigma_u = 0$. }
\end{figure}
\end{enumerate}

When $\sigma_t > 0$, regardless of what resolution $\sigma$ chooses for the rest of the crossings, the term in the state sum are part of the undoing of the $t-1$ kinks in the $\sigma$-state graph. See Figure \ref{f.tg0} for an example. Therefore, 
\[ \deg \sum_{\sigma_t > 0} \langle \sigma \rangle = \deg \ \sigma_1  - 4(t-1).  \]
\begin{figure}[H]
\centering
\def \svgwidth{.3\textwidth}
\begingroup%
  \makeatletter%
  \providecommand\color[2][]{%
    \errmessage{(Inkscape) Color is used for the text in Inkscape, but the package 'color.sty' is not loaded}%
    \renewcommand\color[2][]{}%
  }%
  \providecommand\transparent[1]{%
    \errmessage{(Inkscape) Transparency is used (non-zero) for the text in Inkscape, but the package 'transparent.sty' is not loaded}%
    \renewcommand\transparent[1]{}%
  }%
  \providecommand\rotatebox[2]{#2}%
  \newcommand*\fsize{\dimexpr\f@size pt\relax}%
  \newcommand*\lineheight[1]{\fontsize{\fsize}{#1\fsize}\selectfont}%
  \ifx\svgwidth\undefined%
    \setlength{\unitlength}{482.83043569bp}%
    \ifx\svgscale\undefined%
      \relax%
    \else%
      \setlength{\unitlength}{\unitlength * \real{\svgscale}}%
    \fi%
  \else%
    \setlength{\unitlength}{\svgwidth}%
  \fi%
  \global\let\svgwidth\undefined%
  \global\let\svgscale\undefined%
  \makeatother%
  \begin{picture}(1,0.55927966)%
    \lineheight{1}%
    \setlength\tabcolsep{0pt}%
    \put(0,0){\includegraphics[width=\unitlength,page=1]{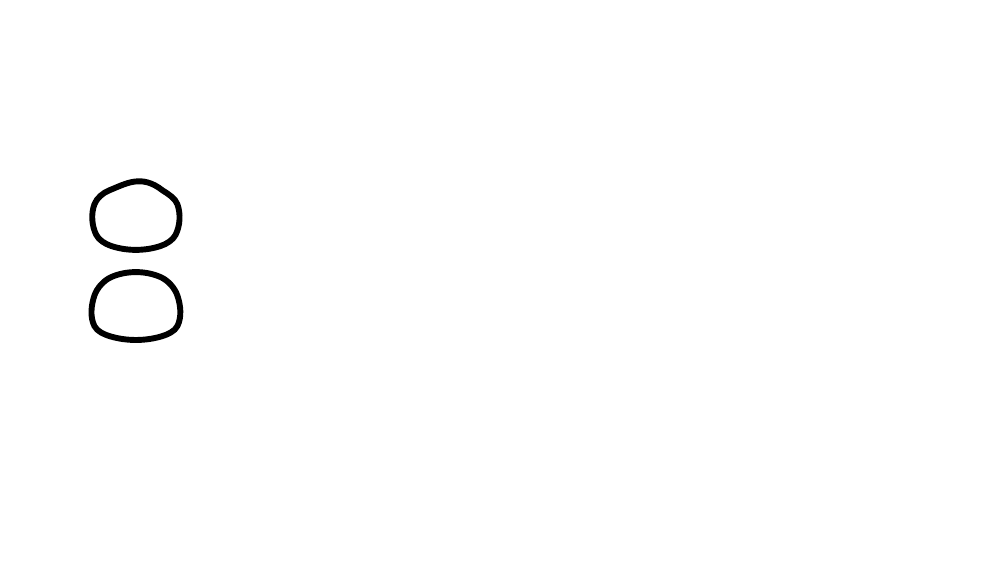}}%
    \put(0.03271092,0.28375298){\color[rgb]{0,0,0}\makebox(0,0)[lt]{\lineheight{1.25}\smash{\begin{tabular}[t]{l}$r$\end{tabular}}}}%
    \put(0.26530737,0.28123674){\color[rgb]{0,0,0}\makebox(0,0)[lt]{\lineheight{1.25}\smash{\begin{tabular}[t]{l}$s$\end{tabular}}}}%
    \put(0.60856772,0.28123674){\color[rgb]{0,0,0}\makebox(0,0)[lt]{\lineheight{1.25}\smash{\begin{tabular}[t]{l}$t-1$\end{tabular}}}}%
    \put(0.8703052,0.27872049){\color[rgb]{0,0,0}\makebox(0,0)[lt]{\lineheight{1.25}\smash{\begin{tabular}[t]{l}$u$\end{tabular}}}}%
    \put(0.48940582,0.105101){\color[rgb]{0,0,0}\makebox(0,0)[lt]{\lineheight{1.25}\smash{\begin{tabular}[t]{l}$v$\end{tabular}}}}%
    \put(0,0){\includegraphics[width=\unitlength,page=2]{te0re0se0.pdf}}%
  \end{picture}%
\endgroup%

\caption{\label{f.tg0} An example where $\sigma_t > 0$. }
\end{figure}
\end{proof}

\subsection{Minimum degree}

For the minimum degree, we similarly distinguish three Kauffman states, the first of which is the all-$B$ Kauffman state and label them $\overline{\sigma_1}, \overline{\sigma_2}, \overline{\sigma_3}$. See Figure \ref{f.othree_Kauffman_states}. 
\begin{figure}[H]
     \centering
     \begin{subfigure}[b]{0.3\textwidth}
         \centering
         \def \svgwidth{\textwidth}
\begingroup%
  \makeatletter%
  \providecommand\color[2][]{%
    \errmessage{(Inkscape) Color is used for the text in Inkscape, but the package 'color.sty' is not loaded}%
    \renewcommand\color[2][]{}%
  }%
  \providecommand\transparent[1]{%
    \errmessage{(Inkscape) Transparency is used (non-zero) for the text in Inkscape, but the package 'transparent.sty' is not loaded}%
    \renewcommand\transparent[1]{}%
  }%
  \providecommand\rotatebox[2]{#2}%
  \newcommand*\fsize{\dimexpr\f@size pt\relax}%
  \newcommand*\lineheight[1]{\fontsize{\fsize}{#1\fsize}\selectfont}%
  \ifx\svgwidth\undefined%
    \setlength{\unitlength}{892.16406634bp}%
    \ifx\svgscale\undefined%
      \relax%
    \else%
      \setlength{\unitlength}{\unitlength * \real{\svgscale}}%
    \fi%
  \else%
    \setlength{\unitlength}{\svgwidth}%
  \fi%
  \global\let\svgwidth\undefined%
  \global\let\svgscale\undefined%
  \makeatother%
  \begin{picture}(1,0.44675204)%
    \lineheight{1}%
    \setlength\tabcolsep{0pt}%
    \put(0,0){\includegraphics[width=\unitlength,page=1]{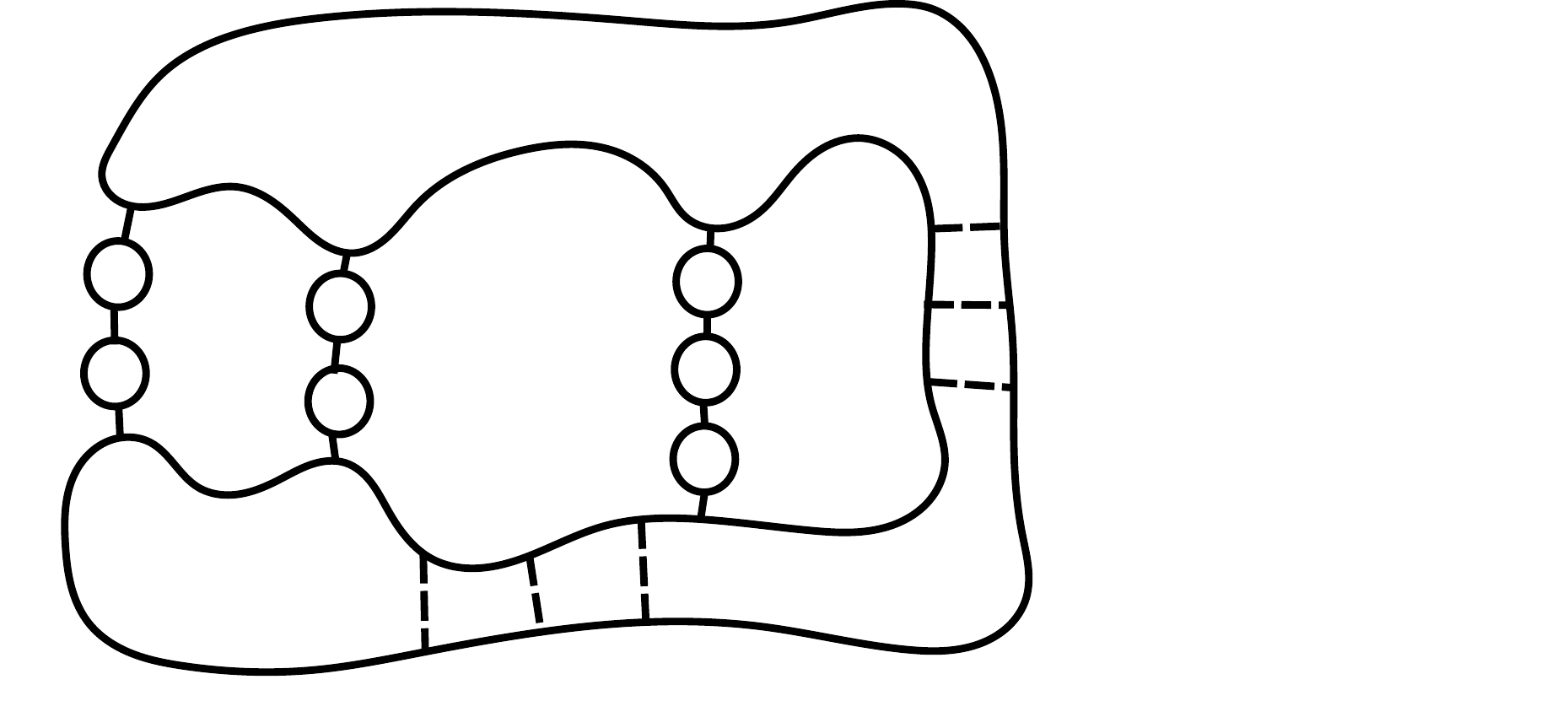}}%
    \put(-0.00203596,0.24767344){\color[rgb]{0,0,0}\makebox(0,0)[lt]{\lineheight{1.25}\smash{\begin{tabular}[t]{l}$r$\end{tabular}}}}%
    \put(0.14140358,0.21410245){\color[rgb]{0,0,0}\makebox(0,0)[lt]{\lineheight{1.25}\smash{\begin{tabular}[t]{l}$s$\\\end{tabular}}}}%
    \put(0.37151722,0.21349209){\color[rgb]{0,0,0}\makebox(0,0)[lt]{\lineheight{1.25}\smash{\begin{tabular}[t]{l}$t$\\\end{tabular}}}}%
    \put(0.65595476,0.22447895){\color[rgb]{0,0,0}\makebox(0,0)[lt]{\lineheight{1.25}\smash{\begin{tabular}[t]{l}$u$\end{tabular}}}}%
    \put(0.31292065,0.00474179){\color[rgb]{0,0,0}\makebox(0,0)[lt]{\lineheight{1.25}\smash{\begin{tabular}[t]{l}$v$\\\end{tabular}}}}%
    \put(0.71646353,0.27411021){\color[rgb]{0,0,0}\makebox(0,0)[lt]{\lineheight{1.25}\smash{\begin{tabular}[t]{l}$\hat{\sigma}_{1}$\end{tabular}}}}%
  \end{picture}%
\endgroup%

         \caption{$\overline{\sigma_1}$}
         \label{f.osigma1}
     \end{subfigure}
     \hfill
     \begin{subfigure}[b]{0.3\textwidth}
         \centering
          \def \svgwidth{\textwidth}
\begingroup%
  \makeatletter%
  \providecommand\color[2][]{%
    \errmessage{(Inkscape) Color is used for the text in Inkscape, but the package 'color.sty' is not loaded}%
    \renewcommand\color[2][]{}%
  }%
  \providecommand\transparent[1]{%
    \errmessage{(Inkscape) Transparency is used (non-zero) for the text in Inkscape, but the package 'transparent.sty' is not loaded}%
    \renewcommand\transparent[1]{}%
  }%
  \providecommand\rotatebox[2]{#2}%
  \newcommand*\fsize{\dimexpr\f@size pt\relax}%
  \newcommand*\lineheight[1]{\fontsize{\fsize}{#1\fsize}\selectfont}%
  \ifx\svgwidth\undefined%
    \setlength{\unitlength}{842.47523991bp}%
    \ifx\svgscale\undefined%
      \relax%
    \else%
      \setlength{\unitlength}{\unitlength * \real{\svgscale}}%
    \fi%
  \else%
    \setlength{\unitlength}{\svgwidth}%
  \fi%
  \global\let\svgwidth\undefined%
  \global\let\svgscale\undefined%
  \makeatother%
  \begin{picture}(1,0.52906176)%
    \lineheight{1}%
    \setlength\tabcolsep{0pt}%
    \put(0,0){\includegraphics[width=\unitlength,page=1]{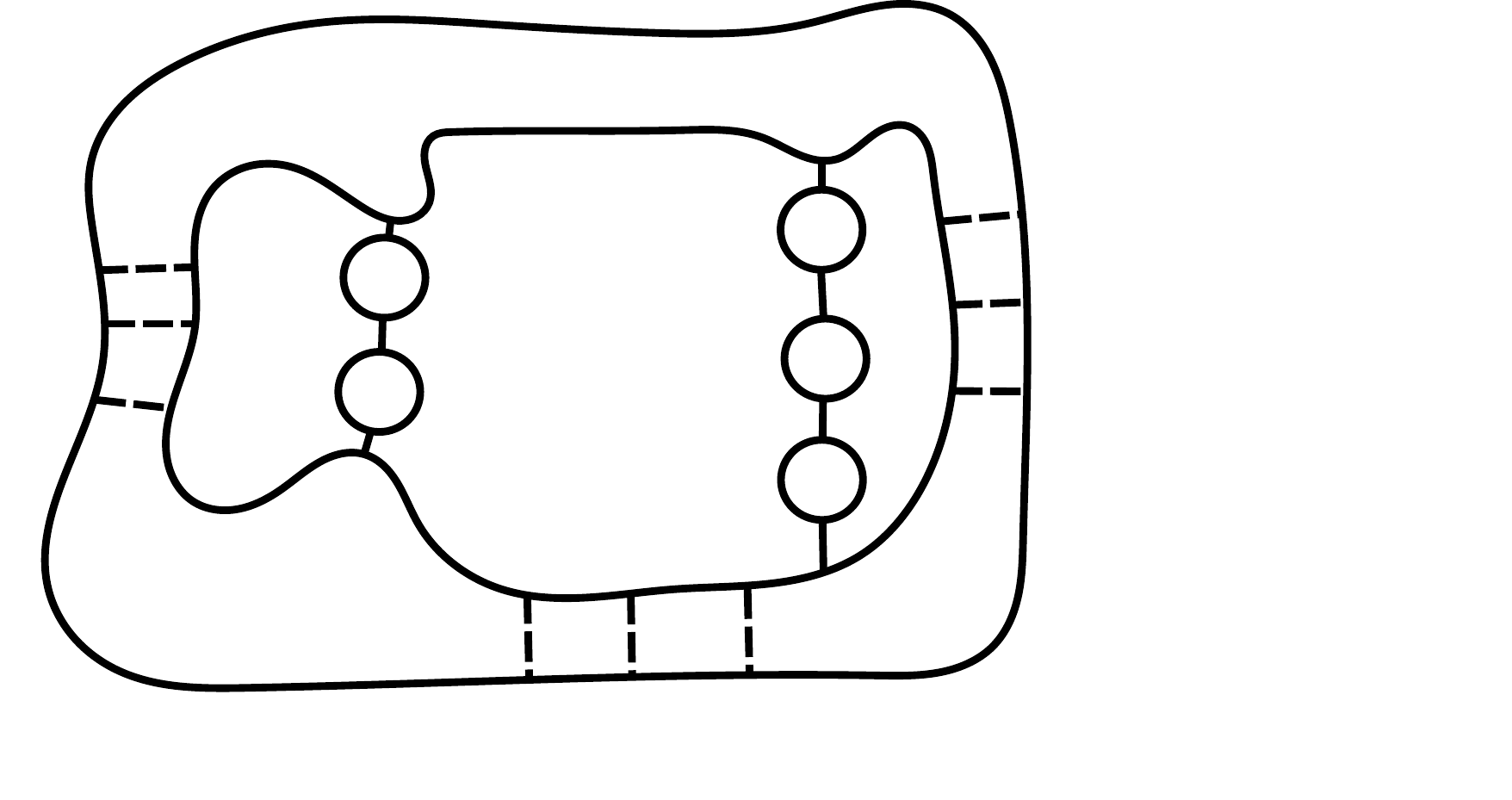}}%
    \put(-0.00215604,0.32298321){\color[rgb]{0,0,0}\makebox(0,0)[lt]{\lineheight{1.25}\smash{\begin{tabular}[t]{l}$r$\end{tabular}}}}%
    \put(0.16528491,0.30340339){\color[rgb]{0,0,0}\makebox(0,0)[lt]{\lineheight{1.25}\smash{\begin{tabular}[t]{l}$s$\end{tabular}}}}%
    \put(0.45227843,0.29703147){\color[rgb]{0,0,0}\makebox(0,0)[lt]{\lineheight{1.25}\smash{\begin{tabular}[t]{l}$t$\end{tabular}}}}%
    \put(0.69282219,0.31659782){\color[rgb]{0,0,0}\makebox(0,0)[lt]{\lineheight{1.25}\smash{\begin{tabular}[t]{l}$u$\end{tabular}}}}%
    \put(0.39338267,0.00502146){\color[rgb]{0,0,0}\makebox(0,0)[lt]{\lineheight{1.25}\smash{\begin{tabular}[t]{l}$v$\end{tabular}}}}%
    \put(0.73525271,0.24460177){\color[rgb]{0,0,0}\makebox(0,0)[lt]{\lineheight{1.25}\smash{\begin{tabular}[t]{l}$\bar{\sigma}_2$\end{tabular}}}}%
  \end{picture}%
\endgroup%

         \caption{$\overline{\sigma_2}$}
         \label{f.osigma2}
     \end{subfigure}
     \hfill
     \begin{subfigure}[b]{0.3\textwidth}
         \centering
         \def \svgwidth{\textwidth}
    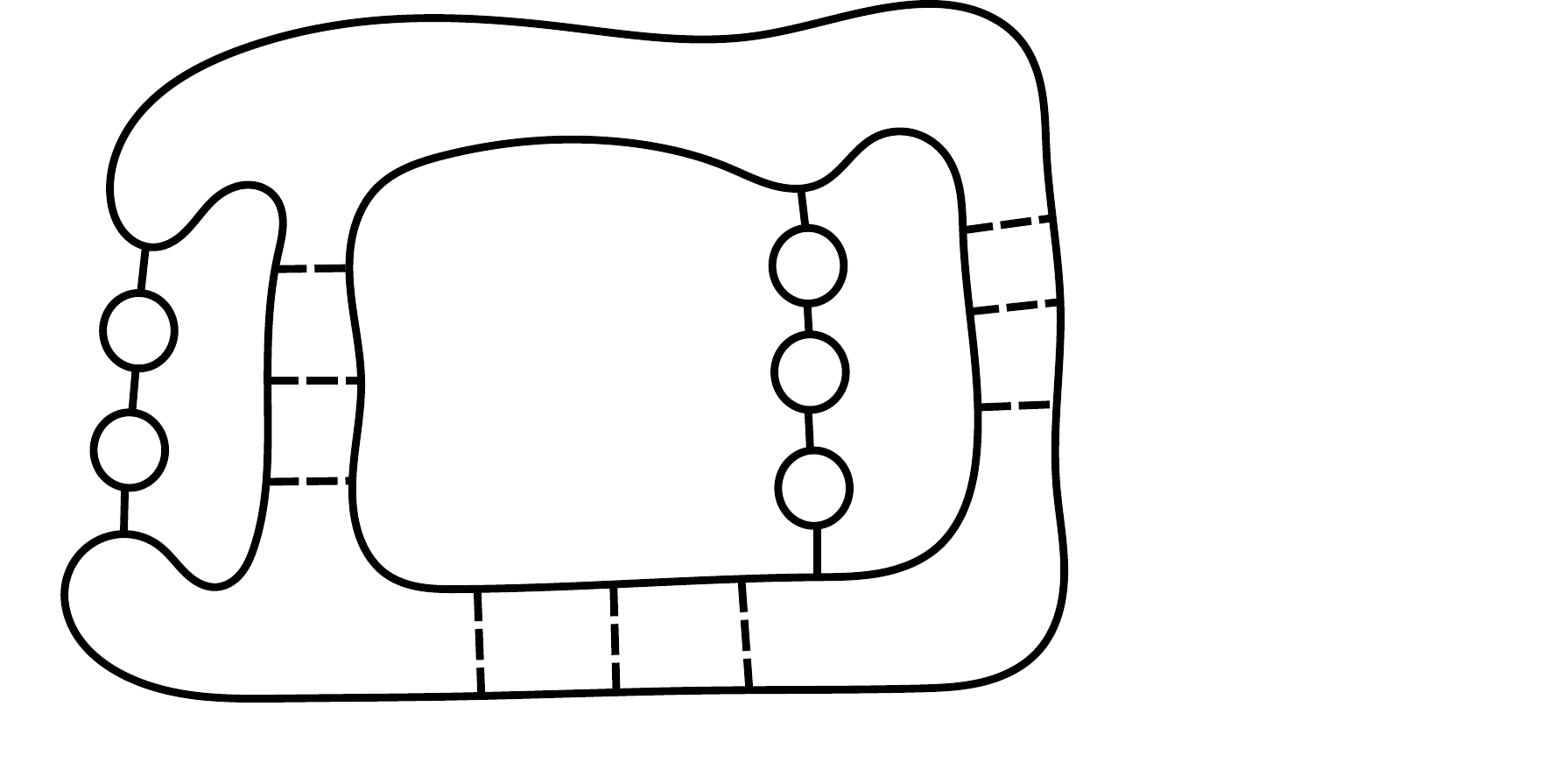
         \caption{$\overline{\sigma_3}$}
         \label{f.osigma_3}
     \end{subfigure}
        \caption{Kauffman states $\overline{\sigma_1}, \overline{\sigma_2}, \overline{\sigma_3}$. }
        \label{f.othree_Kauffman_states}
\end{figure}

\begin{lem} \label{l.vdegreemirror}
   Let $L =  D(r, s, t, u, v)$. We have
    \begin{equation} \label{e.Dstatesum} \langle D(r, s, t, u, v) \rangle = \overline{S_1} + \overline{S_2} + \overline{S_3}, \end{equation}
    where 
    \[ \overline{S_1} = \sum_{\sigma_u + \sigma_v \not= u+v} \langle \sigma \rangle,  \]
    \[ \overline{S_2}  = \sum_{\sigma_u + \sigma_v = u+v, \sigma_r = r, \sigma_s \not= s} \langle \sigma \rangle, \]
    and
    \[ \overline{S_3} = \sum_{\sigma_u + \sigma_v = u+v, \sigma_s = s} \langle \sigma \rangle. \]
    Furthermore, 
    \begin{align*}
        \deg_m \ \overline{S_1} &= \deg_m  \overline{\sigma_1}  + 4(u+v) \\
        &= -6r - 3t+ 6u+4 + 8u  =-6r - 3t+ 14u+4  \\ 
        \deg_m \ \overline{S_2} &= \deg_m  \overline{\sigma_2}   \\
        &= -2r-3t-2u \\
        &= \deg_m \overline{\sigma_1} + 4r-8u-4 \\ 
        \deg_m \ \overline{S_3} &= \deg_m  \overline{\sigma_1} + 4\min
\{u-1, v-1 \}.  \\
    \end{align*}
\end{lem}
\begin{proof}
We decompose the state sum into cases $\sigma_u + \sigma_v = u+v$ and $\sigma_u + \sigma_v < u+v$. The trailing terms (those with the minimum degree of the Kauffman bracket less than or equal to a subset of other terms in the state sum) are organized by their difference with the all-$B$ Kauffman state. 

\paragraph{Case $\sigma_u + \sigma_v = u+v$}
\begin{itemize}
\item $\sigma_r = r$. Note that the $\sigma$-state graph for the state $\overline{\sigma_2}$ which chooses the $B$-resolution on all the rest of the crossings no longer has one-edged loops. Therefore it is the trailing term of the sum. See Figure \ref{f.tg0sigmarer}. 
\begin{figure}[H]
\centering
\def \svgwidth{.3\textwidth}
\begingroup%
  \makeatletter%
  \providecommand\color[2][]{%
    \errmessage{(Inkscape) Color is used for the text in Inkscape, but the package 'color.sty' is not loaded}%
    \renewcommand\color[2][]{}%
  }%
  \providecommand\transparent[1]{%
    \errmessage{(Inkscape) Transparency is used (non-zero) for the text in Inkscape, but the package 'transparent.sty' is not loaded}%
    \renewcommand\transparent[1]{}%
  }%
  \providecommand\rotatebox[2]{#2}%
  \newcommand*\fsize{\dimexpr\f@size pt\relax}%
  \newcommand*\lineheight[1]{\fontsize{\fsize}{#1\fsize}\selectfont}%
  \ifx\svgwidth\undefined%
    \setlength{\unitlength}{482.83043569bp}%
    \ifx\svgscale\undefined%
      \relax%
    \else%
      \setlength{\unitlength}{\unitlength * \real{\svgscale}}%
    \fi%
  \else%
    \setlength{\unitlength}{\svgwidth}%
  \fi%
  \global\let\svgwidth\undefined%
  \global\let\svgscale\undefined%
  \makeatother%
  \begin{picture}(1,0.55927966)%
    \lineheight{1}%
    \setlength\tabcolsep{0pt}%
    \put(0.03271092,0.28375298){\color[rgb]{0,0,0}\makebox(0,0)[lt]{\lineheight{1.25}\smash{\begin{tabular}[t]{l}$r$\end{tabular}}}}%
    \put(0.26530737,0.28123674){\color[rgb]{0,0,0}\makebox(0,0)[lt]{\lineheight{1.25}\smash{\begin{tabular}[t]{l}$s$\end{tabular}}}}%
    \put(0.65637595,0.23908995){\color[rgb]{0,0,0}\makebox(0,0)[lt]{\lineheight{1.25}\smash{\begin{tabular}[t]{l}$t$\end{tabular}}}}%
    \put(0.8703052,0.27872049){\color[rgb]{0,0,0}\makebox(0,0)[lt]{\lineheight{1.25}\smash{\begin{tabular}[t]{l}$u$\end{tabular}}}}%
    \put(0.48940582,0.105101){\color[rgb]{0,0,0}\makebox(0,0)[lt]{\lineheight{1.25}\smash{\begin{tabular}[t]{l}$v$\end{tabular}}}}%
    \put(0,0){\includegraphics[width=\unitlength,page=1]{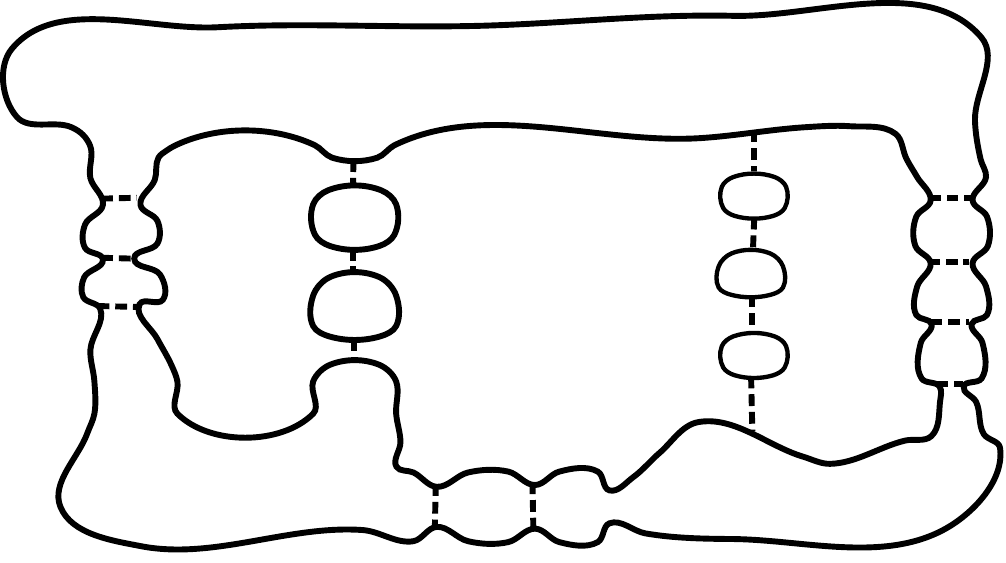}}%
  \end{picture}%
\endgroup%

\caption{\label{f.tg0sigmarer} The $u+v$ one-edged loops in the all-$B$ state no longer connect a single component in the $\sigma$-state graph.}
\end{figure}
In this case the minimum degree is 
\[ \deg_m \sum_{\sigma_u + \sigma_v = u+v, \ \sigma_r = r} \langle \sigma \rangle =  \deg_m  \overline{\sigma}_2  \]
\item $\sigma_s = s$. In this case the $\sigma$-state graph of $\overline{\sigma_3}$ no longer has one-edge loops and thus is the trailing term. See Figure \ref{f.tg0sigmases}.
\begin{figure}[H]
\def \svgwidth{.3\textwidth}
\begingroup%
  \makeatletter%
  \providecommand\color[2][]{%
    \errmessage{(Inkscape) Color is used for the text in Inkscape, but the package 'color.sty' is not loaded}%
    \renewcommand\color[2][]{}%
  }%
  \providecommand\transparent[1]{%
    \errmessage{(Inkscape) Transparency is used (non-zero) for the text in Inkscape, but the package 'transparent.sty' is not loaded}%
    \renewcommand\transparent[1]{}%
  }%
  \providecommand\rotatebox[2]{#2}%
  \newcommand*\fsize{\dimexpr\f@size pt\relax}%
  \newcommand*\lineheight[1]{\fontsize{\fsize}{#1\fsize}\selectfont}%
  \ifx\svgwidth\undefined%
    \setlength{\unitlength}{482.83043569bp}%
    \ifx\svgscale\undefined%
      \relax%
    \else%
      \setlength{\unitlength}{\unitlength * \real{\svgscale}}%
    \fi%
  \else%
    \setlength{\unitlength}{\svgwidth}%
  \fi%
  \global\let\svgwidth\undefined%
  \global\let\svgscale\undefined%
  \makeatother%
  \begin{picture}(1,0.55927966)%
    \lineheight{1}%
    \setlength\tabcolsep{0pt}%
    \put(0.03271092,0.28375298){\color[rgb]{0,0,0}\makebox(0,0)[lt]{\lineheight{1.25}\smash{\begin{tabular}[t]{l}$r$\end{tabular}}}}%
    \put(0.26530737,0.28123674){\color[rgb]{0,0,0}\makebox(0,0)[lt]{\lineheight{1.25}\smash{\begin{tabular}[t]{l}$s$\end{tabular}}}}%
    \put(0.65637595,0.23908995){\color[rgb]{0,0,0}\makebox(0,0)[lt]{\lineheight{1.25}\smash{\begin{tabular}[t]{l}$t$\end{tabular}}}}%
    \put(0.8703052,0.27872049){\color[rgb]{0,0,0}\makebox(0,0)[lt]{\lineheight{1.25}\smash{\begin{tabular}[t]{l}$u$\end{tabular}}}}%
    \put(0.48940582,0.105101){\color[rgb]{0,0,0}\makebox(0,0)[lt]{\lineheight{1.25}\smash{\begin{tabular}[t]{l}$v$\end{tabular}}}}%
    \put(0,0){\includegraphics[width=\unitlength,page=1]{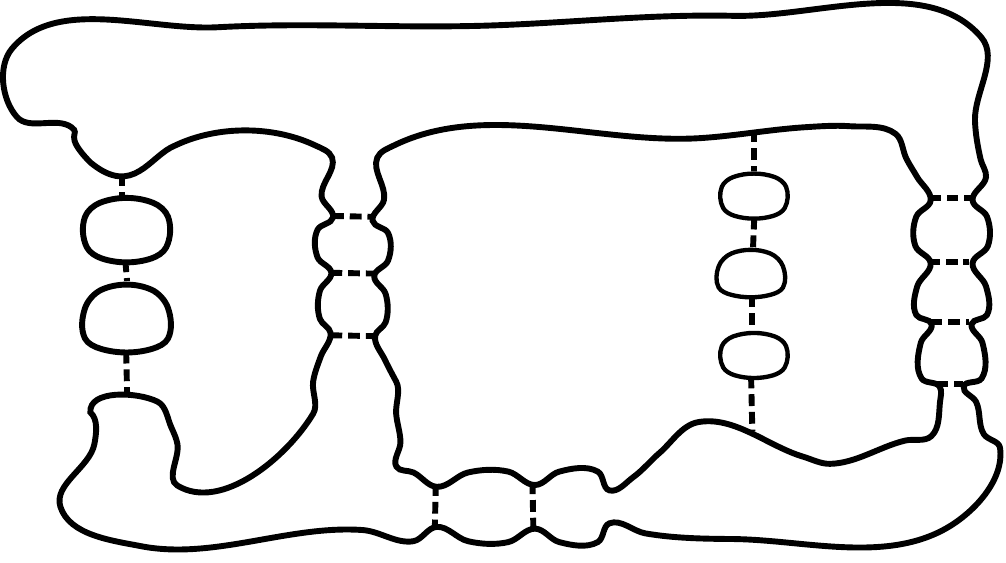}}%
  \end{picture}%
\endgroup%

\caption{\label{f.tg0sigmases}  The $u+v$ one-edged loops in the all-$B$ state no longer connect a single component in the $\sigma$-state graph.}
\end{figure}

The minimum degree is 
\[ \deg_m \sum_{\sigma_u + \sigma_v = u+v, \ \sigma_s = s} \langle \sigma \rangle =  \deg_m  \overline{\sigma}_3  \]
\end{itemize} 
\paragraph{Case $\sigma_u + \sigma_v \not= u+v$}
There is at least one crossing in the twist region $T_u$ or $T_v$ for which the state $\sigma$ chooses the $A$-resolution. This gives either $u-1$ or $v-1$ kinks that can be untwisted via a Reidemeister move that increases the degree by Lemma \ref{l.r1_lowers_degree}. See Figure \ref{f.sigmausigmavneupv} for an example where $\sigma_u =1$ and $\sigma_v =1$. 
\begin{figure}[H]
\centering
\def \svgwidth{.3\textwidth}
\begingroup%
  \makeatletter%
  \providecommand\color[2][]{%
    \errmessage{(Inkscape) Color is used for the text in Inkscape, but the package 'color.sty' is not loaded}%
    \renewcommand\color[2][]{}%
  }%
  \providecommand\transparent[1]{%
    \errmessage{(Inkscape) Transparency is used (non-zero) for the text in Inkscape, but the package 'transparent.sty' is not loaded}%
    \renewcommand\transparent[1]{}%
  }%
  \providecommand\rotatebox[2]{#2}%
  \newcommand*\fsize{\dimexpr\f@size pt\relax}%
  \newcommand*\lineheight[1]{\fontsize{\fsize}{#1\fsize}\selectfont}%
  \ifx\svgwidth\undefined%
    \setlength{\unitlength}{426.51209553bp}%
    \ifx\svgscale\undefined%
      \relax%
    \else%
      \setlength{\unitlength}{\unitlength * \real{\svgscale}}%
    \fi%
  \else%
    \setlength{\unitlength}{\svgwidth}%
  \fi%
  \global\let\svgwidth\undefined%
  \global\let\svgscale\undefined%
  \makeatother%
  \begin{picture}(1,0.67216084)%
    \lineheight{1}%
    \setlength\tabcolsep{0pt}%
    \put(0,0){\includegraphics[width=\unitlength,page=1]{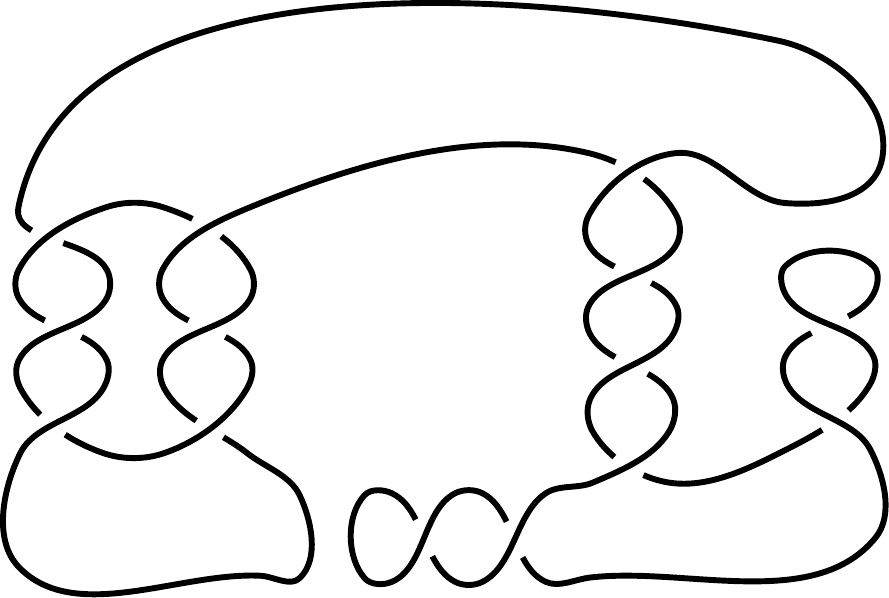}}%
    \put(-0.05121525,0.30037654){\color[rgb]{0,0,0}\makebox(0,0)[lt]{\lineheight{1.25}\smash{\begin{tabular}[t]{l}$r$\end{tabular}}}}%
    \put(0.12814676,0.30037654){\color[rgb]{0,0,0}\makebox(0,0)[lt]{\lineheight{1.25}\smash{\begin{tabular}[t]{l}$s$\end{tabular}}}}%
    \put(0.60628405,0.30749776){\color[rgb]{0,0,0}\makebox(0,0)[lt]{\lineheight{1.25}\smash{\begin{tabular}[t]{l}$t$\end{tabular}}}}%
    \put(0.89380643,0.41071136){\color[rgb]{0,0,0}\makebox(0,0)[lt]{\lineheight{1.25}\smash{\begin{tabular}[t]{l}$u-1$\end{tabular}}}}%
    \put(0.41305394,0.1451925){\color[rgb]{0,0,0}\makebox(0,0)[lt]{\lineheight{1.25}\smash{\begin{tabular}[t]{l}$v-1$\end{tabular}}}}%
  \end{picture}%
\endgroup%

\caption{\label{f.sigmausigmavneupv} This is the case $\sigma_u =1$ and $\sigma_v =1$. }
\end{figure}
\begin{itemize}
\item Suppose $\sigma_u < u$. Then the $\sigma$-state graph has at least $u-1$ one-edged loops. 
\item Suppose $\sigma_v < v$. Then the $\sigma$-state graph has at least $v-1$ one-edged loops. 
\end{itemize}
Thus 
\[ \deg_m \sum_{\sigma_u + \sigma_v \not= u+v} \langle \sigma \rangle \leq  \deg_m \overline{\sigma_1} + 4\min
\{u-1, v-1 \}.  \]
\end{proof}

Now we specify the condition on the parameters $r, s, t, u, v$ that results in the Jones polynomial having leading and trailing coefficients 2 and determine the span of the Jones polynomial. The conditions are that $r=s, u=v$, and $t > \max \{r+1, u+1\}$. We discuss an example of a link $L = D(r, s, t, -u, -v)$ whose parameters $r, s, t, u, v$ satisfy these constraints.

\subsection{The span of the Jones polynomial} \label{ss.spanJones}
We now determine the Maximum and minimum degrees of the Jones polynomial $V_L(t)$ for our family of links. 

\begin{prop}  \label{p.tJonesDegree}
Let $L = D(r, s, t, -u, -v)$ be a link with diagram $D(r, s, t, -u, -v)$ such that $r=s$, $u=v$, and $t \geq \max\{r+2, u+2\}. $ Then the maximum and minimum degrees of the Jones polynomial are 
\begin{align}
\deg V_L(t) &= \frac{-1}{4}(4r-6t-2u+2)   \\ 
\deg_m V_L(t) &= \frac{-1}{4}(8r-2t+2u+2) 
\end{align}
Furthermore, the first and last coefficient of $V_L(t)$ is 2. 
\end{prop}

\begin{proof}
When the parameters $r, s, t, u, v$ satisfy the conditions of the proposition, the degree of the Kauffman bracket of $D$ is equal to $\deg \sigma_2 = 2r+t+2u+4$ by Lemma \ref{l.vdegree}. In particular, the two terms corresponding to the states $\sigma_2, \sigma_3$ realize the maximum degree with the same sign (becaues $r=s$), so the first coefficient of the Kauffman bracket polynomial of $D$ is 2. Therefore, adjusting by the writhe and with the requisite degree conversion from the variable $A$ to the variable $t$ gives:  
\begin{align*} \deg_m V_L(t) &= \frac{-1}{4}(-3w(D) + 2r+t+2u+4-2) \\ 
&= \frac{-1}{4}(-3(-2r+t) + 2r+t+2u+2) \\ 
&= \frac{-1}{4}(8r-2t+2u+2) 
\end{align*}
For the maximum degree we do the same after applying Lemma \ref{l.vdegreemirror}: 
\begin{align*} \deg V_L(t) &= \frac{-1}{4}(-3w(D) -2r-3t-2u+2) \\ 
&= \frac{-1}{4}(-3(-2r+t) -2r-3t-2u+2) \\ 
&= \frac{-1}{4}(4r-6t-2u+2). 
\end{align*}
Again because $u=v$, the two terms corresponding to the Kauffman states $\overline{\sigma_2}$, $\overline{\sigma_3}$ combine to give the last coefficient of the Kauffman bracket of $D$ is 2. With the degree conversion, this gives the first coefficient of the Jones polynomial to be 2. 
\end{proof} 

\begin{cor} \label{c.tJonesDegree}
    Let $L = D(r, s, t, u, v)$. The span of the Jones polynomial is  
    \[ \spn V_L(t) = \frac{1}{4}(4r+4t+4u) = r+t+u . \] 
\end{cor}

\subsection{A link example $L = D(3, 3, 5, -3, -3)$.}
The diagram $D(3, 3, 5, -3, -3)$ is shown in Figure \ref{f.l33533diag}. The Jones polynomial for this link as calculated by SageMath \cite{sage} is 
\begin{align*} 
V_L(t) &= 2t^{11/2} - 4t^{9/2} + 7t^{7/2} - 12t^{5/2} + 14t^{3/2} - 17 t^{1/2} + 16t^{-1/2} - \\ 
&- 13t^{-3/2} + 11t^{-5/2} - 7 t^{-7/2} + 3t^{-9/2} - 2t^{-11/2}.   
\end{align*}

\begin{figure}[H]
\centering
\def \svgwidth{.3\textwidth}
\begingroup%
  \makeatletter%
  \providecommand\color[2][]{%
    \errmessage{(Inkscape) Color is used for the text in Inkscape, but the package 'color.sty' is not loaded}%
    \renewcommand\color[2][]{}%
  }%
  \providecommand\transparent[1]{%
    \errmessage{(Inkscape) Transparency is used (non-zero) for the text in Inkscape, but the package 'transparent.sty' is not loaded}%
    \renewcommand\transparent[1]{}%
  }%
  \providecommand\rotatebox[2]{#2}%
  \newcommand*\fsize{\dimexpr\f@size pt\relax}%
  \newcommand*\lineheight[1]{\fontsize{\fsize}{#1\fsize}\selectfont}%
  \ifx\svgwidth\undefined%
    \setlength{\unitlength}{426.51209553bp}%
    \ifx\svgscale\undefined%
      \relax%
    \else%
      \setlength{\unitlength}{\unitlength * \real{\svgscale}}%
    \fi%
  \else%
    \setlength{\unitlength}{\svgwidth}%
  \fi%
  \global\let\svgwidth\undefined%
  \global\let\svgscale\undefined%
  \makeatother%
  \begin{picture}(1,0.67216084)%
    \lineheight{1}%
    \setlength\tabcolsep{0pt}%
    \put(0,0){\includegraphics[width=\unitlength,page=1]{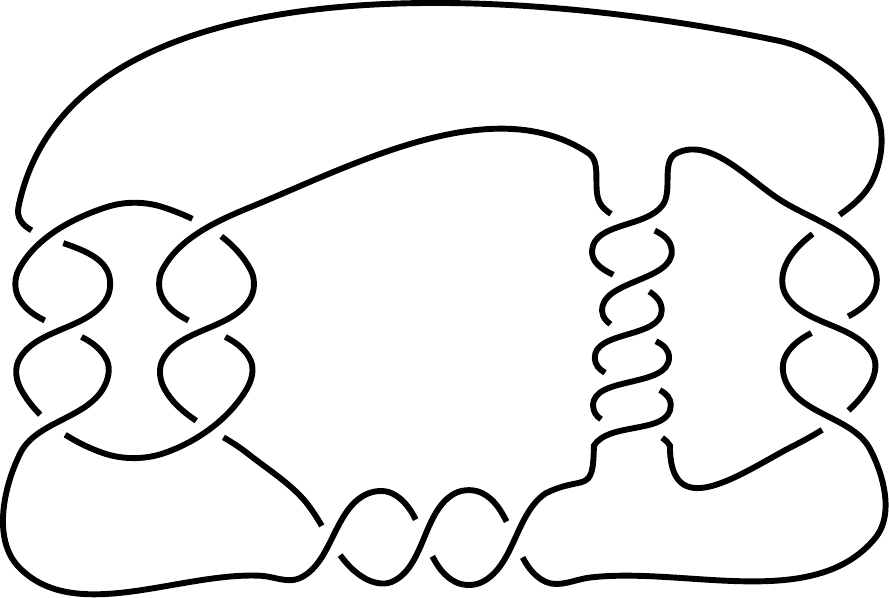}}%
    \put(-0.05121525,0.30037654){\color[rgb]{0,0,0}\makebox(0,0)[lt]{\lineheight{1.25}\smash{\begin{tabular}[t]{l}$r$\end{tabular}}}}%
    \put(0.12814676,0.30037654){\color[rgb]{0,0,0}\makebox(0,0)[lt]{\lineheight{1.25}\smash{\begin{tabular}[t]{l}$s$\end{tabular}}}}%
    \put(0.60628405,0.30749776){\color[rgb]{0,0,0}\makebox(0,0)[lt]{\lineheight{1.25}\smash{\begin{tabular}[t]{l}$t$\end{tabular}}}}%
    \put(0.83610954,0.30546314){\color[rgb]{0,0,0}\makebox(0,0)[lt]{\lineheight{1.25}\smash{\begin{tabular}[t]{l}$u$\end{tabular}}}}%
    \put(0.46580748,0.14167565){\color[rgb]{0,0,0}\makebox(0,0)[lt]{\lineheight{1.25}\smash{\begin{tabular}[t]{l}$v$\end{tabular}}}}%
    \put(0,0){\includegraphics[width=\unitlength,page=2]{diagram2.pdf}}%
  \end{picture}%
\endgroup%

\caption{\label{f.l33533diag} The link $L(3, 3, 5, -3, -3)$ with orientation indicated by the arrows.}
\end{figure}

\subsection{When is $L = D(r, s, t, -u, -v)$ a link and when is it a knot?} \label{ss.linkknotquestion}
We describe how the parity of the number of crossings in each twist region determines the number of components of the link. 
In Table \ref{t.parity_component},  $0 = x \mod 2$ for $x \in \{r, s, t, u, v\}$ and $1 = x \mod 2$ denote parity of $x\in \{r, s, t, u, v\}$. 
\begin{table}[H]
\centering
\begin{tabular}{|r|r|r|r|r|r|l|}
\hline
\multicolumn{1}{|l|}{$r$} & \multicolumn{1}{l|}{$s$} & \multicolumn{1}{l|}{$t$} & \multicolumn{1}{l|}{$u$} & \multicolumn{1}{l|}{$v$} & \multicolumn{1}{l|}{components}  \\ \hline
even & even & even & even & even & 4  \\ \hline
even & even & even & even & odd & 3  \\ \hline
even & even & even & odd & even & 3  \\ \hline
even & even & even & odd & odd & 2  \\ \hline
even & even & odd & even & even & 3  \\ \hline
even & even & odd & even & odd & 2  \\ \hline
even & even & odd & odd & even & 2  \\ \hline
even & even & odd & odd & odd & 2  \\ \hline
even & odd & even & even & even & 3  \\ \hline
even & odd & even & even & odd & 2  \\ \hline
even & odd & even & odd & even & 2  \\ \hline
even & odd & even & odd & odd & 1  \\ \hline
even & odd & odd & even & even & 2  \\ \hline
even & odd & odd & even & odd & 1  \\ \hline
even & odd & odd & odd & even & 1  \\ \hline
even & odd & odd & odd & odd & 1   \\ \hline
odd & even & even & even & even & 3  \\ \hline
odd & even & even & even & odd & 2  \\ \hline
odd & even & even & odd & even & 2 \\ \hline
odd & even & even & odd & odd & 1   \\ \hline
odd & even & odd & even & even & 2  \\ \hline
odd & even & odd & even & odd & 1   \\ \hline
odd & even & odd & odd & even & 1   \\ \hline
odd & even & odd & odd & odd & 1   \\ \hline
odd & odd & even & even & even & 2  \\ \hline
odd & odd & even & even & odd & 2  \\ \hline
odd & odd & even & odd & even & 1  \\ \hline
odd & odd & even & odd & odd & 1  \\ \hline
odd & odd & odd & even & even & 1  \\ \hline
odd & odd & odd & even & odd & 1  \\ \hline
odd & odd & odd & odd & even & 2  \\ \hline
odd & odd & odd & odd & odd & 2  \\ \hline
\end{tabular}
\caption{Parity of crossings in each maximal twist region corresponding to the number of link components. }
\label{t.parity_component}
\end{table}

 \section{Crossing number from the Kauffman polynomial and the Jones polynomial} \label{s.crossingn}
In this section, we show that the given diagram for the family of links $L = D(r, s, t, -u, -v)$ is the minimum crossing diagram. The approach using both the Kauffman polynomial and the Jones polynomial dates from Lickorish-Thistlethwaite \cite{Lickorish-Thistlethwaite}. We first establish some lemmas for our family of links. Compare the following with \cite[Lemma 8]{Lickorish-Thistlethwaite}.

\begin{lem} \label{l.zdegree_dv}
Let $L = D(r, s, t, -u, -1)$, then the $z$-degree of $\Lambda(D)$ is $n-2$, where $n$ is the number of crossings in $D$. 
\end{lem}
\begin{proof}
Applying the skein relation Eq. \eqref{e.lambda} to the single crossing $x$ in the twist region $T_v$ gives
\begin{equation} \label{e.lambda2} \Lambda(D) = -\Lambda(D_v) + z\Lambda(D_0) + z\Lambda(D_\infty).  \end{equation}

The diagram $D_v$ is a non alternating diagram with a bridge of length 2. Therefore it has a $z$-degree $\leq n-3$. 

\begin{figure}[H]
\def \svgwidth{.3\textwidth}
\begingroup%
  \makeatletter%
  \providecommand\color[2][]{%
    \errmessage{(Inkscape) Color is used for the text in Inkscape, but the package 'color.sty' is not loaded}%
    \renewcommand\color[2][]{}%
  }%
  \providecommand\transparent[1]{%
    \errmessage{(Inkscape) Transparency is used (non-zero) for the text in Inkscape, but the package 'transparent.sty' is not loaded}%
    \renewcommand\transparent[1]{}%
  }%
  \providecommand\rotatebox[2]{#2}%
  \newcommand*\fsize{\dimexpr\f@size pt\relax}%
  \newcommand*\lineheight[1]{\fontsize{\fsize}{#1\fsize}\selectfont}%
  \ifx\svgwidth\undefined%
    \setlength{\unitlength}{426.51209553bp}%
    \ifx\svgscale\undefined%
      \relax%
    \else%
      \setlength{\unitlength}{\unitlength * \real{\svgscale}}%
    \fi%
  \else%
    \setlength{\unitlength}{\svgwidth}%
  \fi%
  \global\let\svgwidth\undefined%
  \global\let\svgscale\undefined%
  \makeatother%
  \begin{picture}(1,0.67216084)%
    \lineheight{1}%
    \setlength\tabcolsep{0pt}%
    \put(0,0){\includegraphics[width=\unitlength,page=1]{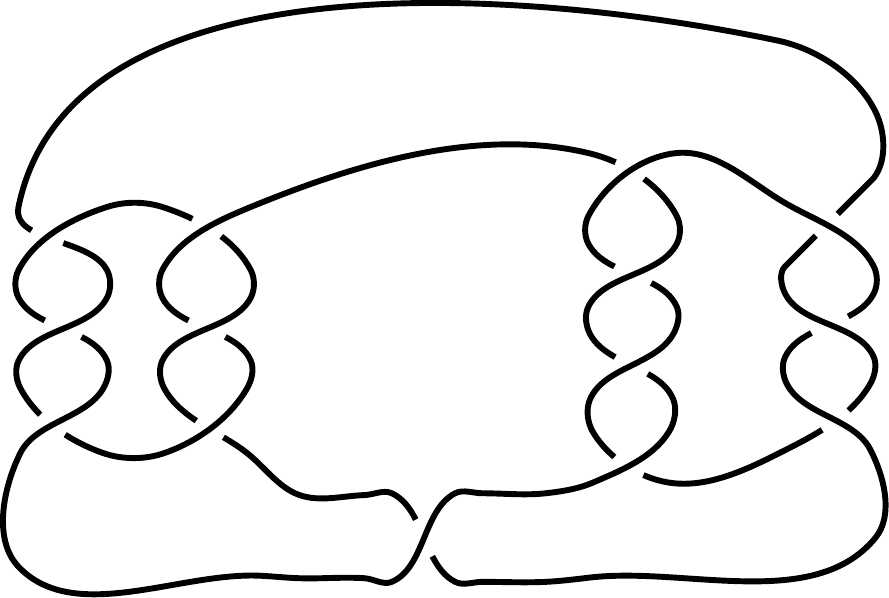}}%
    \put(-0.05121525,0.30037654){\color[rgb]{0,0,0}\makebox(0,0)[lt]{\lineheight{1.25}\smash{\begin{tabular}[t]{l}$r$\end{tabular}}}}%
    \put(0.12814676,0.30037654){\color[rgb]{0,0,0}\makebox(0,0)[lt]{\lineheight{1.25}\smash{\begin{tabular}[t]{l}$s$\end{tabular}}}}%
    \put(0.60628405,0.30749776){\color[rgb]{0,0,0}\makebox(0,0)[lt]{\lineheight{1.25}\smash{\begin{tabular}[t]{l}$t$\end{tabular}}}}%
    \put(0.81852528,0.30546314){\color[rgb]{0,0,0}\makebox(0,0)[lt]{\lineheight{1.25}\smash{\begin{tabular}[t]{l}$u$\end{tabular}}}}%
    \put(0.46580748,0.14167565){\color[rgb]{0,0,0}\makebox(0,0)[lt]{\lineheight{1.25}\smash{\begin{tabular}[t]{l}$1$\end{tabular}}}}%
  \end{picture}%
\endgroup%

\caption{\label{f.Dv} $D_v$.}
\end{figure}

The diagram $D_0$ is the connected sum of two diagrams $L_1 + L_2$. $L_1$ is prime alternating and $L_2$ can be simplified to remove 4 crossings. Therefore, the $z$-degree of $\Lambda(D_0)$ is given by $n-1-4-2 = n-7$. 

\begin{figure}[H]
\centering
\def \svgwidth{.3\textwidth}
\begingroup%
  \makeatletter%
  \providecommand\color[2][]{%
    \errmessage{(Inkscape) Color is used for the text in Inkscape, but the package 'color.sty' is not loaded}%
    \renewcommand\color[2][]{}%
  }%
  \providecommand\transparent[1]{%
    \errmessage{(Inkscape) Transparency is used (non-zero) for the text in Inkscape, but the package 'transparent.sty' is not loaded}%
    \renewcommand\transparent[1]{}%
  }%
  \providecommand\rotatebox[2]{#2}%
  \newcommand*\fsize{\dimexpr\f@size pt\relax}%
  \newcommand*\lineheight[1]{\fontsize{\fsize}{#1\fsize}\selectfont}%
  \ifx\svgwidth\undefined%
    \setlength{\unitlength}{426.51209553bp}%
    \ifx\svgscale\undefined%
      \relax%
    \else%
      \setlength{\unitlength}{\unitlength * \real{\svgscale}}%
    \fi%
  \else%
    \setlength{\unitlength}{\svgwidth}%
  \fi%
  \global\let\svgwidth\undefined%
  \global\let\svgscale\undefined%
  \makeatother%
  \begin{picture}(1,0.67216084)%
    \lineheight{1}%
    \setlength\tabcolsep{0pt}%
    \put(0,0){\includegraphics[width=\unitlength,page=1]{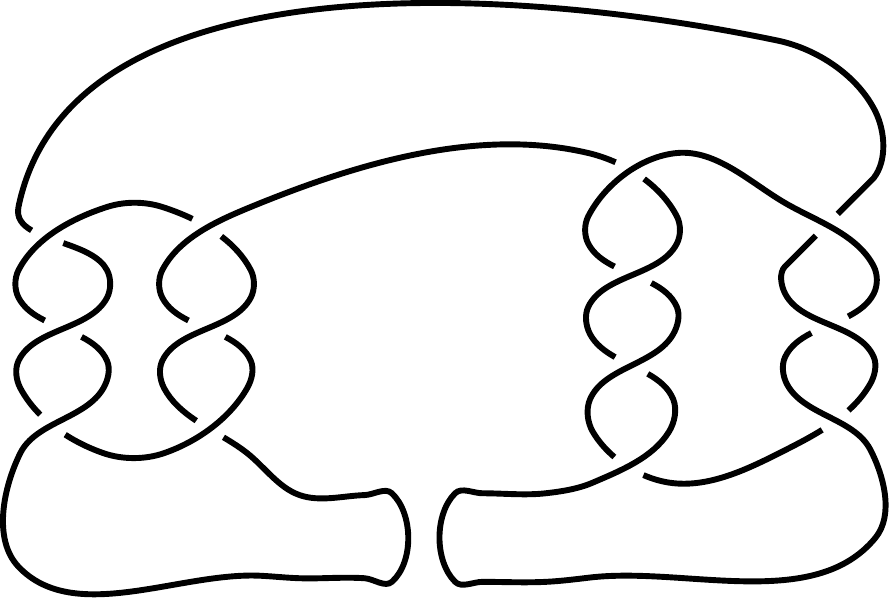}}%
    \put(-0.05121525,0.30037654){\color[rgb]{0,0,0}\makebox(0,0)[lt]{\lineheight{1.25}\smash{\begin{tabular}[t]{l}$r$\end{tabular}}}}%
    \put(0.12814676,0.30037654){\color[rgb]{0,0,0}\makebox(0,0)[lt]{\lineheight{1.25}\smash{\begin{tabular}[t]{l}$s$\end{tabular}}}}%
    \put(0.60628405,0.30749776){\color[rgb]{0,0,0}\makebox(0,0)[lt]{\lineheight{1.25}\smash{\begin{tabular}[t]{l}$t$\end{tabular}}}}%
    \put(0.81852528,0.30546314){\color[rgb]{0,0,0}\makebox(0,0)[lt]{\lineheight{1.25}\smash{\begin{tabular}[t]{l}$u$\end{tabular}}}}%
  \end{picture}%
\endgroup%

\caption{\label{f.D0} $D_0$.}
\end{figure}

The diagram $D_\infty$ is the pretzel knot with a single negative twist region $T_u$. Thus by Lemma \ref{l.nearalt}, the the $z$-degree of $\Lambda(D_\infty)$ is $n-3$. 

\begin{figure}[H]
\centering
\def \svgwidth{.3\textwidth}
\begingroup%
  \makeatletter%
  \providecommand\color[2][]{%
    \errmessage{(Inkscape) Color is used for the text in Inkscape, but the package 'color.sty' is not loaded}%
    \renewcommand\color[2][]{}%
  }%
  \providecommand\transparent[1]{%
    \errmessage{(Inkscape) Transparency is used (non-zero) for the text in Inkscape, but the package 'transparent.sty' is not loaded}%
    \renewcommand\transparent[1]{}%
  }%
  \providecommand\rotatebox[2]{#2}%
  \newcommand*\fsize{\dimexpr\f@size pt\relax}%
  \newcommand*\lineheight[1]{\fontsize{\fsize}{#1\fsize}\selectfont}%
  \ifx\svgwidth\undefined%
    \setlength{\unitlength}{426.51209553bp}%
    \ifx\svgscale\undefined%
      \relax%
    \else%
      \setlength{\unitlength}{\unitlength * \real{\svgscale}}%
    \fi%
  \else%
    \setlength{\unitlength}{\svgwidth}%
  \fi%
  \global\let\svgwidth\undefined%
  \global\let\svgscale\undefined%
  \makeatother%
  \begin{picture}(1,0.67216084)%
    \lineheight{1}%
    \setlength\tabcolsep{0pt}%
    \put(0,0){\includegraphics[width=\unitlength,page=1]{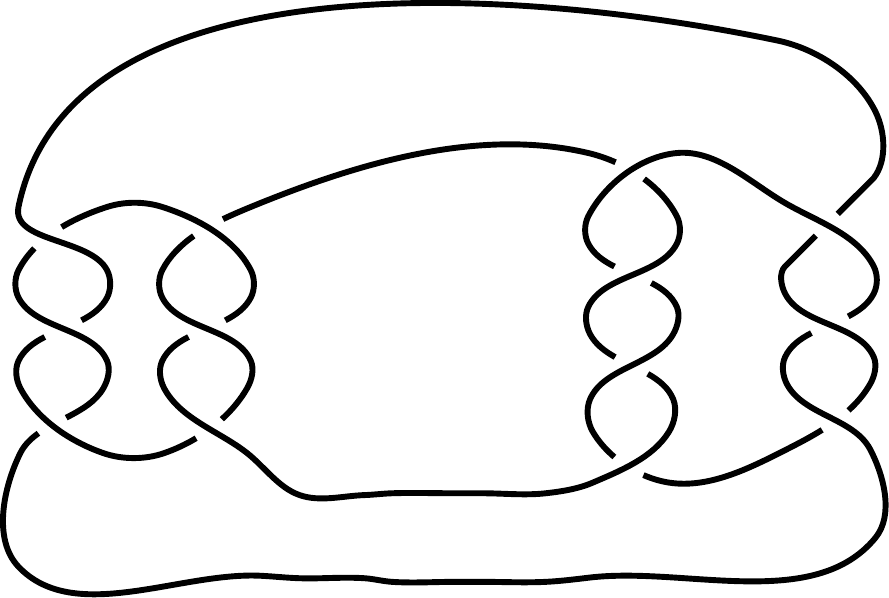}}%
    \put(-0.05121525,0.30037654){\color[rgb]{0,0,0}\makebox(0,0)[lt]{\lineheight{1.25}\smash{\begin{tabular}[t]{l}$r$\end{tabular}}}}%
    \put(0.12814676,0.30037654){\color[rgb]{0,0,0}\makebox(0,0)[lt]{\lineheight{1.25}\smash{\begin{tabular}[t]{l}$s$\end{tabular}}}}%
    \put(0.60628405,0.30749776){\color[rgb]{0,0,0}\makebox(0,0)[lt]{\lineheight{1.25}\smash{\begin{tabular}[t]{l}$t$\end{tabular}}}}%
    \put(0.81852528,0.30546314){\color[rgb]{0,0,0}\makebox(0,0)[lt]{\lineheight{1.25}\smash{\begin{tabular}[t]{l}$u$\end{tabular}}}}%
  \end{picture}%
\endgroup%

\caption{\label{f.Dinf} $D_\infty$.}
\end{figure}

Comparing the $z$-degrees of all the terms in Eq. \eqref{e.lambda2}, we get that the term $z\Lambda(D_{\infty})$ determines the degree of $\Lambda(D_v)$ which is equal to $n-2$. 
\end{proof}

\begin{thm} \label{t.Kauffman_poly_z_degree}
Let $L = D(r, s, t, -u, -v)$ be the $n$ crossing link diagram. Then the $z$-degree of $\Lambda_D(a, z)$ is $n-2$.
\end{thm} 
\begin{proof}
We proceed by induction on the number of half twists in the twist region $T_v$ corresponding to the twist region marked $v$ in Figure \ref{f.diagram}. Apply the skein relation to $D$ to a crossing $x$ of $v$ gives
\[ \Lambda(D) = -\Lambda(D_v) + z\Lambda(D_0) + z\Lambda(D_\infty).  \] 

The effect of switching the crossing $x$ gives a pretzel knot diagram with a single negative twist region, therefore by Lemma \ref{l.nearalt}, the $z$-degree of $\Lambda(D_v)$ is $n-4$. 
The $0$-resolution of the crossing $x$ results in a kink that can be removed without affect the $z$-degree. The result is a connected sum of two prime alternating links, therefore, the $z$-degree of $\Lambda(D_0)$ is $n-4$. Finally, the $\infty$-resolution of the crossing $x$ results in a diagram with a bridge of length 2. Therefore, the $z$-degree of $\Lambda(D_{\infty})$ is $\leq n-3$. Apply Lemma \ref{l.zdegree_dv}, we get that the $z$-degree of $\Lambda(D_{\infty})$ is $n-3$. 

Comparing the $z$-degrees of all the terms in the application of the skein relation gives us that the $z$-degree of $\Lambda(D)$ is $n-2$. 

Now with the induction hypothesis, we assume the number of crossings in the twist region $T_v$ is $k \geq 3$. 

Apply the skein relation again to a crossing $x$ of the twist region $T_v$, we get
\[ \Lambda(D) = -\Lambda(D_v) + z\Lambda(D_0) + z\Lambda(D_\infty).  \] 
Going through a similar process as before, we get that the $z$-degree of $D_v$ is $n-4$. The $z$-degree of $D_0$ is $n-k-2$. The $z$-degree of $D_\infty$ is $n-3$. Multiplying by $z$ we get that  $ z\Lambda(D_\infty)$ determines the degree of $\Lambda(D)$ which is $n-2$.

\end{proof}

\subsection{The minimum crossing number of $L=D(r, s, t, -u, -v)$.}

\begin{thm} \label{t.minCrossing}
The diagram $D$ is a minimum-crossing diagram. Therefore, $c(L) = r+s+t+u+v$. 
\end{thm}

\begin{proof}
    We show that the link $L$ cannot be projected with fewer than $n$ crossings unless it admits a $n-2$ crossing prime alternating diagram $D'$. Suppose that it admits an $n-1$ crossing diagram that is not alternating and not prime, then its $z$-degree for $\Lambda(L)$ would be less than or equal to $n-1-1-2  = n-4$, a contradiction to Lemma \ref{l.zdegree_dv}. Similarly, we can rule out $L$ admitting an $n-1$ crossing diagram that is not alternating and prime and arrive at the case where it admits an $n-1$ crossing prime alternating diagram. 

    We also know that the link cannot be projected with fewer than $n-3$ crossings becaues of the $z$-degree bound. Hence we also arrive at a possibility that it admits an $n-2$ crossing prime alternating diagram. 
    However, this means that $\spn(V_L(t)) = n-2$. By Proposition \ref{p.tJonesDegree}, $\spn(V_L) = r+t+u < (n-2) = r+s+t+u+v-2$, hence $D$ realizes the minimum number crossing number of $L$. 
\end{proof}

\section{The Turaev genus for the family $L=D(r, s, t, -u, -v)$}
We apply the following theorem from \cite{DL} characterizing the extreme coefficients of the Jones polynomial for Turaev genus one links. 
\begin{thm}\cite[Theorem 1.3]{DL}
    Let $L$ be an almost alternating link or a link with Turaev genus one with Jones polynomial 
\[ J_L(t) = a_mt^m + a_{m+1}t^{m+1} + \cdots + a_{M-1} t^{M-1} + a_M t^M, \]
where $a_m$ and $a_M$ are nonzero. Either $|a_m| = 1$ or $|a_M| = 1$ (or both). 
\end{thm}

\begin{cor} \label{c.TG} If $L$ is a link with a diagram $D$ such that $g_T(D) = 2$ and both $a_m(L), a_M(L) \not=1$, then $L$ has Turaev genus 2. That is, 
\[ g_T(L) = g_T(D).\]
\end{cor}

 Now we are ready to prove Theorem \ref{t.defect}, which we restate here as a proposition for the convenience of the reader. 
\begin{prop}  \label{p.defect}
Let $L = D(r, s, t, -u, -v)$, then 
\[ \delta(L) = c(L) - g_T(L) - \spn(V_L(t)) = s+v - 2.  \]
\end{prop}
\begin{proof}
    The proposition follows from combining the results of Theorem \ref{t.minCrossing} for $c(L)$, Corollary \ref{c.TG} for $g_T(L)$, and Corollary \ref{c.tJonesDegree} for $\spn(V_L(t))$. 
\end{proof}

\section{Diagram and moves that decrease Turaev genus} \label{s.ddecreasetg}

It is well known that any two diagrams of the same link are related by the Reidemeister moves.  We illustrate and label them for the results of this section in Figure. \ref{f.reidemeistermoves}
\begin{figure}[H]
\centering
\def \svgwidth{.7\textwidth}
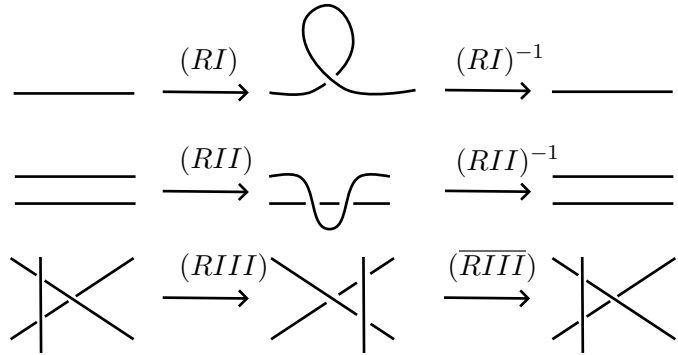
\caption{\label{f.reidemeistermoves} Reidemeister moves.}
\end{figure}

We now consider the question of understanding the link diagrams of the same link that could result in a decrease of Turaev genus. In particular, we let $\D^2$ be a disk in the projection plane of a link diagram $D$ of $L$, whose boundary intersects the link diagram transversely in finitely many points $D\cap \partial \D^2$ on which a Reidemeister move can be performed. That is, in the disk $\D^2$ a link looks like one of the pictures in Figure \ref{f.reidemeistermoves}. 

In his thesis, Lowrance \cite{Lowrance-thesis} computed the possibilities of how Reidemeister moves could affect the Turaev genus. We take this idea further and consider the closures of the link diagram in the disk $\D^2$. That is, the possible matchings of $L\cap \partial \D^2$  that could result in a decrease of the Tureav genus. This gives us information of the possible local features of the all-$A$ state graph of a link for which the Turaev genus could decrease through a Reidemeister move. 

We will designate the tangle of a link diagram $D$ in a small disk $\D^2$ in which a Reidemeister move is performed by $D_T$, and mark the points of $\D^2\cap D$ with letter in $ \{a, b, c, d, e, f\}$.

We now prove Theorem \ref{t.rmoveseffect},  which we reprint here for the convenience of the reader.

\begin{prop} \label{p.rmoveseffect}
    Let $D$ be a diagram for the link $L$ and let $D'$ be the knot diagram obtained from $D$ after a performing a Reidemeister move on $D_T$ in a small disk $\D^2$. If $g_T(D') < g_T(D)$, then the Reidemeister move must be inverse Type II, or of Type III. Moreover, there are certain conditions on the closures of the link diagram around the move in the A- or B-state: 
    \begin{itemize}
    \item[$\mathrm{(RII)}^{-1}$] If the Reidemeister move is of inverse Type II, then marking $D\cap \partial\D^2$ with $a, b, c, d$, we have, for the initial diagram $D$, the closure as illustrated in Figure \ref{f.r2-1matching}. 
    \begin{figure}[H]
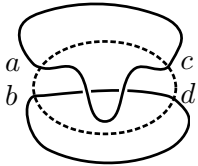

    \centering
    \def \svgwidth{.2\textwidth}
        
    \caption{\label{f.r2-1matching} Matching of $a, b, c, d$ in $D$ when the Reidemeister move is $\mathrm{(RII)}^{-1}$.}
     \end{figure}    
    \item[$\mathrm{(RIII)}/\mathrm{(\overline{RIII})}$] If the Reidemeister move is of Type III, then the possible closures of the ends of the tangle $D_T$ marked by $a, b, c, d, e, f$ are illustrated in Figure \ref{f.r3closure}. 
    \begin{figure}[H]
    \centering
    \def \svgwidth{.7\textwidth}
    \input{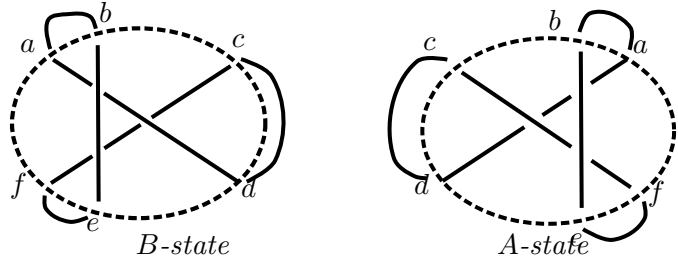} 
    \caption{\label{f.r3closure} Left: Matching of $a, b, c, d, e, f$ when the Reidemeister move is $\mathrm{RIII}$. Right: Matching of $a, b, c, d, e, f$ when the Reidemeister move is $\mathrm{\overline{RIII}}$. }
    \end{figure}
    \end{itemize}
    \end{prop}

\begin{proof}
    We will first show that Reidemeister Type I and non-inverse Type II moves do not decrease the Turaev genus, regardless of the diagram $D$. 

    A Reidemeister Type I move adds one crossing and one circle to either the $A$-state or the $B$-state. 
    \begin{figure}[H]
    \centering
    \def \svgwidth{.5\textwidth}
\begingroup%
  \makeatletter%
  \providecommand\color[2][]{%
    \errmessage{(Inkscape) Color is used for the text in Inkscape, but the package 'color.sty' is not loaded}%
    \renewcommand\color[2][]{}%
  }%
  \providecommand\transparent[1]{%
    \errmessage{(Inkscape) Transparency is used (non-zero) for the text in Inkscape, but the package 'transparent.sty' is not loaded}%
    \renewcommand\transparent[1]{}%
  }%
  \providecommand\rotatebox[2]{#2}%
  \newcommand*\fsize{\dimexpr\f@size pt\relax}%
  \newcommand*\lineheight[1]{\fontsize{\fsize}{#1\fsize}\selectfont}%
  \ifx\svgwidth\undefined%
    \setlength{\unitlength}{424.52504075bp}%
    \ifx\svgscale\undefined%
      \relax%
    \else%
      \setlength{\unitlength}{\unitlength * \real{\svgscale}}%
    \fi%
  \else%
    \setlength{\unitlength}{\svgwidth}%
  \fi%
  \global\let\svgwidth\undefined%
  \global\let\svgscale\undefined%
  \makeatother%
  \begin{picture}(1,0.80473477)%
    \lineheight{1}%
    \setlength\tabcolsep{0pt}%
    \put(0,0){\includegraphics[width=\unitlength,page=1]{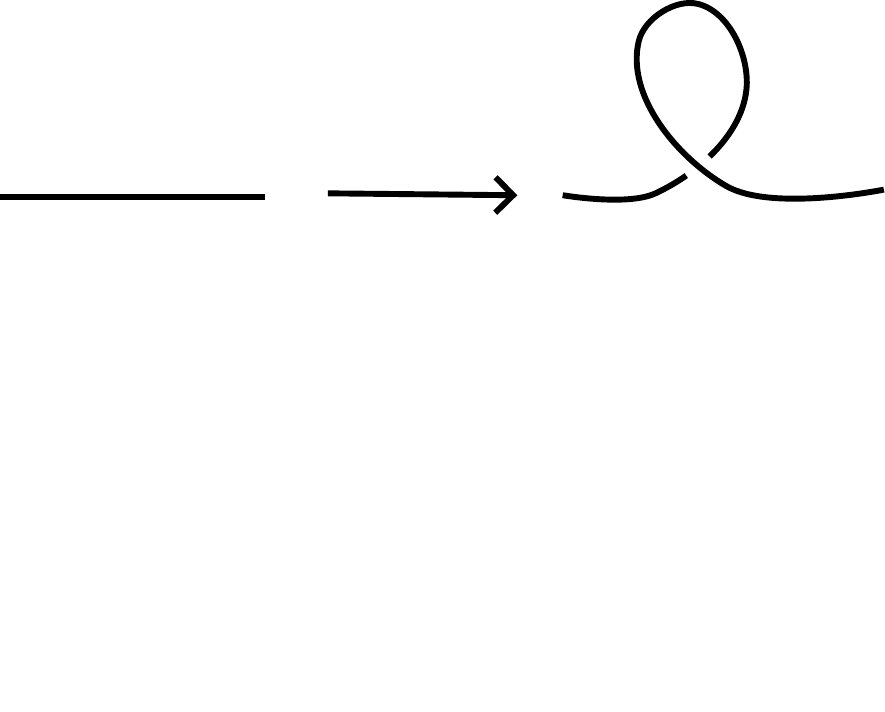}}%
    \put(0.41144947,0.64281401){\color[rgb]{0,0,0}\makebox(0,0)[lt]{\lineheight{1.25}\smash{\begin{tabular}[t]{l}$(RI)$\end{tabular}}}}%
    \put(0,0){\includegraphics[width=\unitlength,page=2]{r1effect.pdf}}%
    \put(0.41144947,0.37427835){\color[rgb]{0,0,0}\makebox(0,0)[lt]{\lineheight{1.25}\smash{\begin{tabular}[t]{l}$(RI)$\end{tabular}}}}%
    \put(0,0){\includegraphics[width=\unitlength,page=3]{r1effect.pdf}}%
    \put(0.41144947,0.08100936){\color[rgb]{0,0,0}\makebox(0,0)[lt]{\lineheight{1.25}\smash{\begin{tabular}[t]{l}$(RI)$\end{tabular}}}}%
    \put(0,0){\includegraphics[width=\unitlength,page=4]{r1effect.pdf}}%
  \end{picture}%
\endgroup%

    \caption{A Reidmeister I move adds a crossing and a circle to either the $A$-state or the $B$-state.}
    \end{figure}

    In Figure \ref{f.r2abstate}, $(RI)$ is the same direction as the arrow and $(RI)^{-1}$ is the same direction as the reverse arrow. 
    Hence 
    \begin{align*}
    g_T(D') &= \frac{c(D') + 2 - |s_A(D')|-|s_B(D')|}{2} \\
    &=\frac{(c(D)+1)+2-|s_A(D)|-|s_B(D)|-1}{2} =g_T(D).
    \end{align*}

    A Reidemeister Type II move adds two crossings and at most one circle each to the A-state and the B-state. See Figure \ref{f.r2abstate}. 
   \begin{figure}[H]
   \centering
   \def \svgwidth{.5\textwidth}
\begingroup%
  \makeatletter%
  \providecommand\color[2][]{%
    \errmessage{(Inkscape) Color is used for the text in Inkscape, but the package 'color.sty' is not loaded}%
    \renewcommand\color[2][]{}%
  }%
  \providecommand\transparent[1]{%
    \errmessage{(Inkscape) Transparency is used (non-zero) for the text in Inkscape, but the package 'transparent.sty' is not loaded}%
    \renewcommand\transparent[1]{}%
  }%
  \providecommand\rotatebox[2]{#2}%
  \newcommand*\fsize{\dimexpr\f@size pt\relax}%
  \newcommand*\lineheight[1]{\fontsize{\fsize}{#1\fsize}\selectfont}%
  \ifx\svgwidth\undefined%
    \setlength{\unitlength}{396.26305768bp}%
    \ifx\svgscale\undefined%
      \relax%
    \else%
      \setlength{\unitlength}{\unitlength * \real{\svgscale}}%
    \fi%
  \else%
    \setlength{\unitlength}{\svgwidth}%
  \fi%
  \global\let\svgwidth\undefined%
  \global\let\svgscale\undefined%
  \makeatother%
  \begin{picture}(1,0.64791374)%
    \lineheight{1}%
    \setlength\tabcolsep{0pt}%
    \put(0,0){\includegraphics[width=\unitlength,page=1]{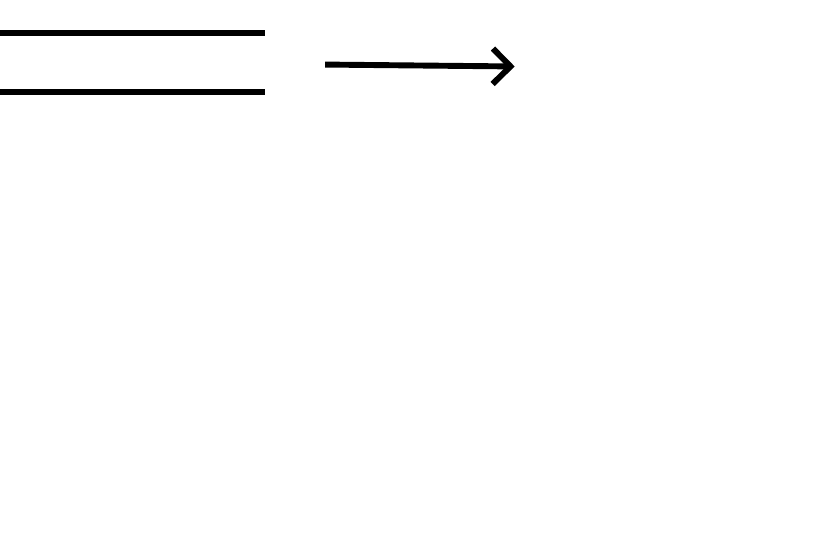}}%
    \put(0.43761406,0.63043601){\color[rgb]{0,0,0}\makebox(0,0)[lt]{\lineheight{1.25}\smash{\begin{tabular}[t]{l}$(RII)$\end{tabular}}}}%
    \put(0,0){\includegraphics[width=\unitlength,page=2]{r2effect.pdf}}%
    \put(0.43761406,0.42224083){\color[rgb]{0,0,0}\makebox(0,0)[lt]{\lineheight{1.25}\smash{\begin{tabular}[t]{l}$(RII)$\end{tabular}}}}%
    \put(0,0){\includegraphics[width=\unitlength,page=3]{r2effect.pdf}}%
    \put(0.43761406,0.18376386){\color[rgb]{0,0,0}\makebox(0,0)[lt]{\lineheight{1.25}\smash{\begin{tabular}[t]{l}$(RII)$\end{tabular}}}}%
    \put(0,0){\includegraphics[width=\unitlength,page=4]{r2effect.pdf}}%
  \end{picture}%
\endgroup%

   \caption{\label{f.r2abstate} The effect of a Reidemeister II move on the $A$-state and the $B$-state.}
   \end{figure}
    
    Hence
    \begin{align*}
    g_T(D') &= \frac{c(D') + 2 - |s_A(D')|-|s_B(D')|}{2} \\ 
    & \geq \frac{(c(D)+2)+2-|s_A(D)|-|s_B(D)|-2}{2} =g_T(D).
    \end{align*}

\vspace{2em}

    Now we consider the inverse Reidemeister Type II move. See Figure \ref{f.inversetype2} for an illustration of the possible closures of $D_T$. 
    \begin{figure}[htp]
    \centering
    \def \svgwidth{0.85\textwidth}
    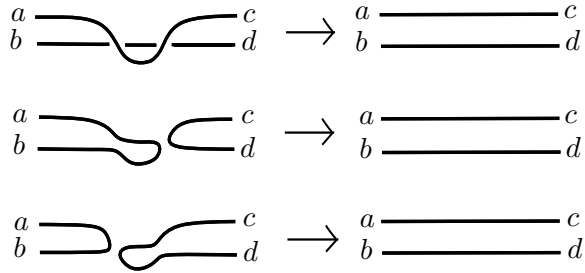
    \caption{\label{f.inversetype2} How the $A$-state and $B$ state change under the Inverse Reidemeister Type II move.}
\end{figure}
Since the $A$-state and $B$-state are planar, there are only two possible closures in each. Namely,with labeling as in the diagrams in Figure \ref{f.type2closures},
\begin{enumerate}[(i)]
    \item $a$ connects to $b$, $c$ connect to $d$; or
    \item $a$ connects to $c$, $b$ connects to $d$.
\end{enumerate}
\begin{figure}[htp]
    \centering
     \def \svgwidth{0.85\textwidth}
    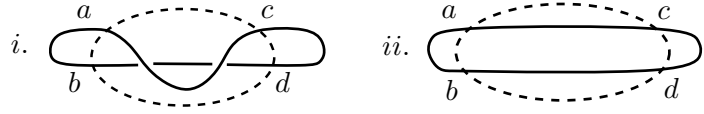
    \caption{\label{f.type2closures} Type II Closures. }
\end{figure}

If the closure is as described in (i), then the move decreases the number of circles in the state by one. If the closure is as in (ii), then the move increases the number of circles in the state by one. Therefore, we have that
\[s_A(D') \geq s_A(D) -1, \qquad s_B(D') \geq s_B(D) -1,\]
with equality if and only if the closure is as in (i). Applying the formula for the Turaev genus of a diagram, we get
\begin{align*}
    g_T(D') &= \frac{c(D') -2 - |s_A(D')|-|s_B(D')|}{2} \\ 
    &\leq\frac{(c(D)+1)-2-(|s_A(D)|-1)-(|s_B(D)|-1)}{2}  =g_T(D), 
    \end{align*}
with equality if and only if the closures in both states are as in (i). Therefore, the inverse Reidemeister Type II move decreases the Turaev genus when at least one of either the A- or B-state has closure as in (ii). 

It remains to consider the Reidemeister Type III move. 
\begin{figure}[H]
    \centering
    \def \svgwidth{0.9\textwidth}
    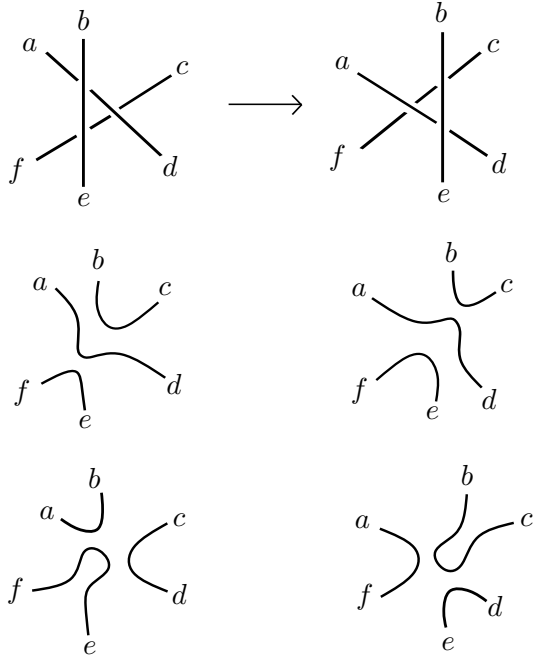
    
    \caption{Reidemeister Type III.}
\end{figure}

We will first enumerate the possible closures. The endpoint $a$ cannot be connected to either $c$ or $e$ since $b$ or $f$ would be an isolated end which does not occur in an all-A state graph, while the $A$-state of the link diagram is a set of simple closed curves. Hence $a$ connects to either $b$, $d$, or $f$.

If $a$ connects to $b$ then either $c$ connects to $d$ or $f$. If $a$ connects to $d$, then $b$ connects to $c$, and $e$ connects to $f$. If $a$ connects to $f$, then either $b$ connects to $c$ or $e$. Hence the possible closures are as listed below from $(i)$ to $(v)$.
\begin{enumerate}[(i)]
    \item $a$ connects to $b$, $c$ connects to $d$, $e$ connects to $f$; 
    \item $a$ connects to $b$, $c$ connects to $f$, $d$ connects to $e$; 
    \item $a$ connects to $d$, $b$ connects to $c$, $e$ connects to $f$; 
    \item $a$ connects to $f$, $b$ connects to $c$, $d$ connects to $e$; and
    \item $a$ connects to $f$, $b$ connects to $e$, $c$ connects to $d$.
\end{enumerate}


Note that the $A$-state of $D$ and $D'$ are identical. Hence $|s_A(D')|=|s_A(D)|$ for all closures. For $|s_B(D')|$, by inspection it can be checked that for the cases (i) through (v), we have 
\begin{enumerate}[(i)]
    \item $|s_B(D)| = 3, \ |s_B(D')| = 1; $
    \item $|s_B(D)| = |s_B(D')| = 2; $
    \item $|s_B(D)| = |s_B(D')| = 2; $ 
    \item $|s_B(D)| = 1, \ |s_B(D')| = 3; $ and
    \item $|s_B(D)| = |s_B(D')| = 2.$
\end{enumerate}

Therefore the Turaev genus of $D'$ is given by
\begin{align*}
g_T(D') &= \frac{c(D') -2 - |s_A(D')|-|s_B(D')|}{2}\\ &=\frac{c(D)-2-|s_A(D)|-|s_B(D)|+|s_B(D)|-|s_B(D')|}{2}\\ & =g_T(D) + \frac{|s_B(D)|-|s_B(D')|}2.
\end{align*}
Therefore the Turaev genus decreases when 
\[\frac{|s_B(D)|-|s_B(D')|}2 < 0,\]
which only occurs in case (iv). 
\end{proof}

\section*{Acknowledgments}
We wish to thank the Mathworks summer program at Texas State University for making this collaboration possible, and we are grateful to Adam Lowrance for conversations about his work. C. Lee was partially supported by NSF grants DMS 1907010 (University of South Alabama), DMS 2244923 (Texas State University) and CAREER-DMS 2440680 (Texas State University).
\bibliographystyle{alpha}
\bibliography{references}

\end{document}